\documentclass[11pt]{article}
\pdfoutput=1

\usepackage{url}
\usepackage{mathtools}
\usepackage{amssymb,bm}
\usepackage{amsthm}
\usepackage{empheq}
\usepackage{latexsym}
\usepackage{enumitem}
\usepackage{eurosym}
\usepackage{dsfont}
\usepackage{appendix}
\usepackage{color} 
\usepackage[unicode]{hyperref}
\usepackage{frcursive}
\usepackage[utf8]{inputenc}
\usepackage[T1]{fontenc}
\usepackage{geometry}
\usepackage{multirow}
\usepackage{todonotes}
\usepackage{lmodern}
\usepackage{anyfontsize}
\usepackage{stmaryrd}
\usepackage{natbib}
\usepackage{cleveref}
\usepackage{comment}
\usepackage{lipsum}

\usepackage{tcolorbox}
\setcitestyle{numbers,open={[},close={]}}

\definecolor{red}{rgb}{0.7,0.15,0.15}
\definecolor{green}{rgb}{0,0.5,0}
\definecolor{blue}{rgb}{0,0,0.7}
\hypersetup{colorlinks, linkcolor={red},citecolor={green}, urlcolor={blue}}
			
\makeatletter \@addtoreset{equation}{section}

\newtheorem{theorem}{Theorem}[section]
\newtheorem{assumption}[theorem]{Assumption}
\newtheorem{corollary}[theorem]{Corollary}
\newtheorem{example}[theorem]{Example}

\newtheorem{lemma}[theorem]{Lemma}
\newtheorem{proposition}[theorem]{Proposition}

\newtheorem{definition}[theorem]{Definition}
\newtheorem{remark}[theorem]{Remark}

\DeclareUnicodeCharacter{014D}{\=o}
\newcommand{\smallfont}[1]{\text{\fontsize{4}{4}\selectfont$#1$}}

\def \E{\mathbb{E}}
\def \F{\mathbb{F}}

\def \K{\mathbb{K}}
\def \L{\mathbb{L}}

\def \N{\mathbb{N}}

\def \P{\mathbb{P}}
\def \Q{\mathbb{Q}}
\def \R{\mathbb{R}}

\def \X{\mathbb{X}}
\def \Y{\mathbb{Y}}
\def \Z{\mathbb{Z}}

\def\Ac{{\cal A}}
\def\Bc{{\cal B}}
\def\Cc{{\cal C}}

\def\Fc{{\cal F}}
\def\Gc{{\cal G}}

\def\Kc{{\cal K}}

\def\Mc{{\cal M}}
\def\Nc{{\cal N}}

\def\Pc{{\cal P}}

\def\Wc{{\cal W}}
\def\Xc{{\cal X}}
\def\Yc{{\cal Y}}
\def\Zc{{\cal Z}}

\usepackage{mathrsfs}
\usepackage{graphicx}

\usepackage{amsmath}

\newcommand{\FPclass}{\mathcal{FP}}
\newcommand{\FP}{{\rm FP}}
\newcommand{\proj}{{\rm proj}}
\newcommand{\cpl}{{\rm cpl}}
\newcommand{\cplbc}{{\rm cpl}_{\rm bc}}
\newcommand{\cplc}{{\rm cpl}_{\rm c}}
\newcommand{\cplmc}{{\rm cpl}_{\rm mc}}
\newcommand{\AW}{\mathcal{AW}}

\author{Beatrice {\sc Acciaio} \footnote{ETH Z\"urich, Department of Mathematics, Switzerland, beatrice.acciaio@math.ethz.ch} \and Daniel {\sc Kr\v{s}ek} \footnote{ETH Z\"urich, Department of Mathematics, Switzerland, daniel.krsek@math.ethz.ch} \and Gudmund {\sc Pammer} \footnote{TU Graz, Institute of Statistics, Austria, gudmund.pammer@tugraz.at}}

\title{Multicausal transport: barycenters and dynamic matching}

\date{\today}	
\begin{document}
\newcommand\independent{\protect\mathpalette{\protect\independenT}{\perp}}
\def\independenT#1#2{\mathrel{\rlap{$#1#2$}\mkern2mu{#1#2}}}

\maketitle

\begin{abstract} 
We introduce a multivariate version of causal transport, which we name \emph{multicausal transport}, involving several filtered processes among which causality constraints are imposed. Subsequently, we consider the barycenter problem for stochastic processes with respect to causal and bicausal optimal transport, and study its connection to specific multicausal transport problems. Attainment and duality of the aforementioned problems are provided. As an application, we study a matching problem in a dynamic setting where agent types evolve over time. We link this to a causal barycenter problem and thereby show existence of equilibria.
\vspace{0.5cm}


\end{abstract}

\section{Introduction}

In recent years, causal optimal transport has emerged as a suitable counterpart to the classical optimal transport theory when considering stochastic processes.
To account for the evolution of information over time, causality constraints are imposed on the set of admissible transport plans.
In financial modeling, stochastic processes are used to describe the evolution of asset prices, in particular also reflecting beliefs of financial agents. In a market with multiple agents, then, a natural question is that of finding a `consensus model', which is our main motivation for studying barycenter problems with respect to causal and bicausal transport distances.
The latter can be reformulated as a multimarginal analogue of the bicausal transport problem which we name multicausal transport.

\medskip In the current setting of stochastic processes, the bicausal analogue of the classical Wasserstein distance is known as adapted Wasserstein distance.
In our analysis, we adopt the setting of \citeauthor*{BaBePa21} \cite{BaBePa21}, where the general framework for adapted Wasserstein distance for filtered processes was considered.
Wide applications in mathematical finance, stochastic optimal control, and machine learning suggest that the adapted Wasserstein distance is the
correct distance in order to obtain stability; see \emph{e.g.}\
\cite{Al81,He96,PfPi12,BaBaBeEd19a,XuLiMuAc20,BaBePa21,BaWi23,HoLeLiLySa23}.
Furthermore, it is remarkable that the topology induced by this distance coincides with multiple other topologies originating in different areas of mathematics, see \cite{BaBaBeEd19b,pa22}. The present paper further extends the theory and can be divided into three intertwined parts, which we briefly introduce below.

\medskip 
{\bf Multicausal transport} 

\smallskip The multimarginal formulation of standard optimal transport has been studied, among others, by \citeauthor*{Pa14} \cite{Pa14}, \citeauthor*{KiPa13} \cite{KiPa13} and \citeauthor*{Gangbo_MultimOT} \cite{Gangbo_MultimOT}. In the present work, we consider a generalization of causal transport maps to a situation with several marginals, which we coin \emph{multicausal}. We study the corresponding optimal transport problem, show its attainment, dynamic programming principle, and a duality result. Given a family of stochastic processes, such a problem can be interpreted as a way to measure how `variable' this family is. 
We also refer to the work \cite{KrPa23}, where multicausal transport naturally appears in a specific problem of robust pricing and hedging. Moreover, we provide a link between certain multicausal problems and bicausal barycenter problems.

\medskip {\bf Causal and bicausal baryenters} 

\smallskip  Barycenters of probability measures in Wasserstein space are believed to be an appropriate way of averaging them while taking into account the geometric properties of the underlying space. Introduced by \citeauthor*{AgCa11} \cite{AgCa11}, this problem has become remarkably popular for its applications in statistics, machine learning, and other fields. Interested readers can refer, for example, to \citeauthor*{CuDo14} \cite{CuDo14}, \citeauthor*{KiPa17} \cite{KiPa17}, \citeauthor*{LeLo17} \cite{LeLo17}, \citeauthor*{PaZe19} \cite{PaZe19}, \citeauthor*{RaPeDeBe12} \cite{RaPeDeBe12}, and the references therein for more details.
Given several filtered processes and cost functions, we study causal and bicausal barycenters. These can be seen as an appropriate way of averaging stochastic processes if one aims to also account for the flow of information in time. 

\medskip We show that the barycenter problems are attained under appropriate assumptions, and study the dual problems.
Moreover, it has already been noted in \cite{AgCa11} that the Wasserstein barycenter problem is closely related to a suitable  multimarginal transport problem. We provide an analogous result in the dynamic framework for the bicausal barycenter and the multicausal optimal transport. Such a reformulation not only offers a possible way to compute barycenters, but is also valuable from a theoretical point of view. More specifically, it provides insights into the structure of the barycenters and potentially offers an alternative way of showing attainment. Furthermore, we show that such a result, in general, does not hold for the causal barycenter problem due to its asymmetric nature.

\medskip Whether causal or bicausal barycenters should be considered depends on the context.
In the case of decision-making under uncertainty in financial markets, we may find ourselves in a situation where we have different models or `beliefs' (for example, provided by different agents in the market) and want to find a way to average them to build a `consensus model' that merges the properties of these models into one in an optimal way. This can be achieved by considering the adapted Wasserstein barycenter, which offers two clear advantages. 
On the one hand, the adapted Wasserstein barycenter of martingales remains in the class of martingales (so that averaging in this way between possible pricing models still produces a meaningful pricing model), which is not the case when taking averages with respect to classical Wasserstein distances. On the other hand, as the adapted Wasserstein distance is robust with respect to a broad class of stochastic optimization problems (such as hedging, pricing, and utility maximization; see \cite{BaBaBeEd19a}), using this distance makes sense in a financial context, as the average model then performs similarly to the initial models in all such problems.

\medskip A different scenario where the causal barycenter may instead be the appropriate choice is when we are looking for a consensus model in a situation where each agent has their own filtration, and the goal is to agree on a single distribution that describes the dynamics of the asset. In such a problem, the solution is a process with a canonical filtration (\emph{i.e.}, we only care about the law of the process), which, however, still depends on the filtrations of the given models. Many problems in mathematical finance remain robust even in this setting; see \cite{AcBaZa20}. We also provide an application in matching models in a dynamic setting, which offers further interpretation and insights into this formulation.

\medskip {\bf Dynamic matching models}

\smallskip  As our last contribution, we introduce the problem of matching in dynamic setting. The connections of optimal transport and matching models in economics have been investigated by numerous authors, see, for example, \citeauthor*{GaSa22} \cite{GaSa22}, \citeauthor*{Dupuy_2014} \cite{Dupuy_2014}, \citeauthor*{Ga16} \cite{Ga16}, \citeauthor*{ChSa16} \cite{ChSa16}, \citeauthor*{CaObOu15} \cite{CaObOu15}. We point out the work of \citeauthor*{Carlier_matching} \cite{Carlier_matching}, who studied the problem of matching models and its connections to barycenters in Wasserstein distance. Motivated by \cite{Carlier_matching}, we consider a situation in which agents divided into several groups are hired by principals to work on specific tasks. In such a setting, with a contract in place, every individual maximizes their utility dynamically in time based on the evolution of some random process, which represents their type evolving in time, thus facing a dynamic stochastic optimal control problem. Moreover, since the type is assumed to be unobserved by others and thus cannot be contracted upon, we face an adverse selection situation. We define a suitable form of equilibria in this setting and, exploiting the results of our work as well as \citeauthor*{KrPa23} \cite{KrPa23}, we show that such an equilibrium exists under appropriate conditions.

\medskip The remainder of the work is organized as follows. \Cref{sec:notation} introduces the notation used throughout the article, as well as the main definitions. In \Cref{sec:multicausal}, we study multicausal optimal transport, its attainment and dual problem. \Cref{sec:AWbary,sec:CWbary} present the main results concerning causal and bicausal barycenter problems. In \Cref{sec:matching_models}, we introduce the problem of matching models in a dynamic setting and provide an existence result for equilibria.

\section{Setting and preliminaries}\label{sec:notation} 
The set of positive integers is denoted by $\N$ and the set of reals by $\R.$ 
Let $(\Omega,\Fc)$ be a measurable space. 
We denote by $\L^0(\Fc)$ the set of all $\R$-valued, $\Fc$-measurable random variables. The set of all probability measures on $(\Omega,\Fc)$ is denoted by $\Pc(\Omega,\Fc)$. If the $\sigma$-algebra $\Fc$ is obvious from the context, \emph{e.g.} if $\Omega$ is a Polish space and $\Fc=\Bc(\Omega)$, we will simplify the notation to $\Pc(\Omega).$  If $\Omega$ is a Polish space, we endow $\Pc(\Omega)$ with the weak topology.
If $\P \in \Pc(\Omega, \Fc)$ is a probability measure and $p\geq 1$, we write $\L^p(\Fc,\P)$ for the set of functions in $\L^0(\Fc)$ with finite $p$-th moment under $\P$. 
Whenever $\Fc$ is clear from the context we shall write $\L^p(\P)$ instead of $\L^p(\Fc,\P).$

\medskip Let $f : \Omega \longrightarrow \Yc$ be a measurable map, where $\Yc$ is a Polish space and $\P \in \Pc(\Omega,\Fc).$ We denote by $f_{\#} \P \in \Pc(\Yc)$ the push-forward of $\P$ under $f.$ That is to say, the probability measure  determined by \[ \int G(y) (f_{\#} \P)(\mathrm{d}y)= \int G(f(\omega)) \P(\mathrm{d} \omega),\quad G \in \Cc_b(\Yc), \] where $\Cc_b(\Yc)$ denotes the space of all real-valued continuous and bounded functions on $\Yc.$ If $\Ac \subseteq \Pc(\Omega,\Fc)$ is a subset of probability measures, we say that a statement holds $\Ac$--quasi-surely, abbreviated to $\Ac$--q.s., if it holds everywhere outside of a set $A \in \Fc$ that satisfies $\P(A)=0$ for every $\P \in \Ac.$

\medskip Throughout this work, we fix $N \in \N$ and a time horizon $T \in \N.$ If $x^i \in A^i$, $i \in \{1,\ldots,N\}$, are given elements of some sets $A^i$, we will make use of the shorthand notation $x^{1:N} \coloneqq (x^1,\ldots,x^N) \in A^{1:N}\coloneqq\prod_{i=1}^N A^i.$ Analogously, if $m,n \in \{1,\ldots,N\}$ are such that $m \leq n$, we define $x^{m:n} \coloneqq (x^m,x^{m+1},\ldots,x^n) \in A^{m:n}\coloneqq \prod_{i=m}^n A^i.$ In the same way, if $x_t \in A_t$, $t \in \{1,\ldots,T\}$, for some sets $A_t$, we denote $x_{1:T} \coloneqq (x_1,\ldots,x_T) \in A_{1:T} \coloneqq \prod_{t=1}^T A_t$, \emph{etc}.

\medskip
Let $\Xc \coloneqq \prod_{t = 1}^T \Xc_t$ be a path space, where each $(\Xc_t,d_t)$ is a separable complete metric space. We equip $\Xc$ with the metric $d(x,y) \coloneqq \sum_{t = 1}^T d_t(x_t,y_t)$, which renders it a complete metric space, hence a Polish space. A filtered process $\X$ is a $5$-tuple
\[
    \X = \big( \Omega^\X, \Fc^\X, \F^\X, \P^\X, X \big),
\]
where $(\Omega^\X, \Fc^\X, \F^\X, \P^\X)$ is a filtered probability space with $\F^\X=(\Fc_t^\X)_{t = 1}^T$ and $X= (X_t)_{t = 1}^T$, is an $\F^\X$-adapted process, where $X_t$ takes values in $\Xc_t$, for $t \in \{1,\ldots,T\}$.

\medskip
Throughout the article, we operate under the following \emph{standing assumption} without further explicit mention.

\begin{assumption} \label{assum_space} For every filtered process $\X$ we have $\Omega^\X=\Omega^\X_{1:T}$ for some $\Omega^\X_t,\; t \in \{1,\ldots,T\}.$ Further, $\Omega^\X_t$ are Polish spaces endowed with their Borel $\sigma$-algebras and $\F^\X=(\Fc_t^{\X})_{t=1}^T$ is the canonical filtration on the space $\Omega^\X_{1:T}.$ That is, $\Fc_t^\X=\bigotimes_{s=1}^t \Bc(\Omega^\X_{s}) \otimes \bigotimes_{s=t+1}^T \{ \emptyset, \Omega^\X_s \}.$  Finally, we have that $\Fc^\X = \Fc^\X_T.$
\end{assumption}
We use the convention $\Fc_0^\X= \{ \Omega^\X,\emptyset \}$ and denote the class of all filtered process satisfying \Cref{assum_space} by $\FPclass$.
 
\medskip 
For a filtered process $\X$,
by disintegrating $\P^\X$ in successive kernels, we write 
\[ \P^\X(\mathrm{d}\omega_{1:T})=\P^\X_{1}(\mathrm{d}\omega_1) \otimes\P^\X_{2,\omega_\smallfont{1}}(\mathrm{d}\omega_2) \otimes \cdots \otimes \P^\X_{T,\omega_{\smallfont{1}:\smallfont{T}\smallfont{-}\smallfont{1}}}(\mathrm{d}\omega_T).\] 

The following assumption is also assumed throughout the paper without further explicit mentioning.

\begin{assumption} \label{ass:continuity} For every filtered process $\X,$ we have that the maps \[X: \Omega^{\X} \longrightarrow \Xc\quad \text{and}\quad\omega_{t-1}\longmapsto \P^\X_{t,\omega_{\smallfont{t}\smallfont{-}\smallfont{1}}}(\mathrm{d}\omega_t),\;t \in \{2,\ldots,T\}\] are continuous.
\end{assumption} 
We emphasize that \Cref{ass:continuity} is without loss of generality, as we can always refine the topology on $\Omega^{\X}$ using \Cref{lem:topology}. 
Moreover, both Assumptions~\ref{assum_space} and \ref{ass:continuity} are satisfied by the canonical representative of $\X$, see \cite[Definition 3.8]{BaBePa21}. 
We also note that for most results in this work replacing $\X$ by its canonical representative is without loss of generality.

\medskip 
Throughout this article we will deal with $N$-tuples $(\X^1, \ldots, \X^N) \in \FPclass^N$ and write
\[ \X^i=\big( \Omega^i, \Fc^i, \F^i, \P^i, X^i \big),\quad i \in\{1,\ldots,N\}. \]
Further, products of filtrations and probability spaces are denoted with an overline, \emph{i.e.},
\begin{align*}
    \overline{\Omega} \coloneqq \prod_{i = 1}^N \Omega^i, \quad \overline{\Fc}_{v} \coloneqq \bigotimes_{i = 1}^N \Fc^i_{v_i}, \quad \overline{\Fc}_t \coloneqq \bigotimes_{i = 1}^N \Fc^i_t,
\end{align*}
where $v \in \{0,\ldots,T\}^N$ and $t \in \{0,\ldots,T\}$. The $i$-th unit vector in $\R^N$ is denoted by $e_i$.
Using this notation, we have for example that $\overline{\Fc}_{t e_1} = \Fc_t^1 \otimes \Fc_0^2 \otimes \ldots \otimes \Fc_0^N$, $t \in \{ 0,\ldots, T \}$. We further set $\overline{\F} \coloneqq (\overline{\Fc}_{t})_{t=0}^T.$ We shall represent a generic element of $\Omega^i_t$ by $\omega_t^i$ and that of $\Omega^i$ by $\omega^i$.

\medskip For $I \subseteq \{1,\ldots, N\}$ we denote by ${\rm proj}^I : \overline{\Omega} \longrightarrow \prod_{i \in I} \Omega^i$ the projection map on the coordinates in $I$. We write ${\rm proj}^i$ instead of ${\rm proj}^{\{i \}}$ when $I = \{ i \}.$ Similarly, we denote by $\proj^{1:t} : \Omega^i \longrightarrow \Omega^i_{1:t}$ the projection on the first $t$ coordinates, \emph{etc}. 
In correspondence with the above, we write $\overline{\omega}_t\coloneqq(\omega_t^1,\ldots, \omega_t^N)$ and $\overline{\omega}_{1:t}\coloneqq(\omega_{1:t}^1,\ldots, \omega_{1:t}^N)$ for a generic element of $\overline{\Omega}_t\coloneqq \prod_{i=1}^N \Omega^i_t$ and $\overline{\Omega}_{1:t}\coloneqq \prod_{s=1}^t \overline{\Omega}_{s}$, respectively.
If $f^i: \Omega^i \longrightarrow \Yc$ is a measurable map for some Polish space $\Yc$, we will formally identify $f^i$ with the map $f^i \circ {\rm proj}^i : \overline{\Omega} \longrightarrow \Yc$ and write $f^i(\overline{\omega})=f^i(\omega^i).$ 
Moreover, if $f^i$ is $\Fc_t^i$-measurable, we will sometimes write $f^i(\overline{\omega})=f^i(\omega^i)=f^i(\omega_{1:t}^i).$
In particular, we have $X^i_t(\overline{\omega})=X^i_t(\omega^i)=X^i_t(\omega^i_{1:t})$.

\begin{remark} In the special case when $\Omega^i=\prod_{t=1}^T \Omega^i_t$ coincides with $\Xc=\prod_{t=1}^T \Xc_t$ for some $i \in\{ 1,\ldots,N\}$, the notation becomes natural. We have $\Omega_{1:t}^i=\Xc_{1:t}\coloneqq \Xc_1 \times \cdots \times \Xc_t$, \emph{etc}. That is to say, every $\omega^i_{1:t} \in \Omega^i_{1:t}$ corresponds to a path $x_{1:t} \in \Xc_{1:t}.$ Indeed, if $f : \Xc \longrightarrow \R$ is $\Fc^i_t$-measurable, it is natural to write $f(x)=f(x_{1:t})$ to emphasize the dependency. 
\end{remark}

\begin{definition}
    Let $N\geq 2$ and $(\X^1,\ldots, \X^N) \in \FPclass^N$.
    We denote by $\cpl(\X^1,\ldots, \X^N)$ the set of probability measures $\pi$ on the measurable space $(\overline{\Omega}, \overline{\Fc})$ that satisfy $\proj^i_{\#} \pi = \P^i$, $i \in \{1,\ldots,N\},$ and call its elements couplings. When $N = 2$, $\pi \in \cpl(\X^1,\X^2)$ is called:
    \begin{enumerate}[label = (\roman*)]
        \item causal if, for all $t \in \{1,\ldots,T\}$, we have under $\pi$ that $\overline{\Fc}_{(T,0)}$ is conditionally independent of $\overline{\Fc}_{(0,t)}$ given $\overline{\Fc}_{(t,0)}$;
        \item bicausal if $\pi$ and $e_\# \pi$ are both causal, where $e: \Omega^1 \times \Omega^2 \longrightarrow \Omega^2 \times \Omega^1$ with $e(\omega^1,\omega^2) \coloneqq (\omega^2,\omega^1)$.
    \end{enumerate}
    For general $N \ge 2$, $\pi \in \cpl(\X^1,\ldots, \X^N)$ is called:
    \begin{enumerate}[label = (\roman*), resume]
        \item multicausal if, for all $t \in \{1,\ldots, T\}$ and $i \in \{1,\ldots,N\}$, we have under $\pi$ that $\overline{\Fc}_{Te_i}$ is conditionally independent of $\overline{\Fc}_t$ given $\overline{\Fc}_{te_i}$. 
    \end{enumerate}
    We denote the set of all causal, bicausal and multicausal couplings by $\cplc(\X^1,\X^2), \cplbc(\X^1,\X^2)$ and $\cplmc(\X^1,\ldots,\X^N)$, respectively.
\end{definition}

Clearly, the notions of multicausality and bicausality coincide for $N=2$, while the case $N > 2$ is a generalization of the bicausal case to a multimarginal situation, where causality is imposed among all marginals.
For $p\geq 1$, the $p$-adapted Wasserstein distance between two elements $(\X^1,\X^2) \in \FPclass^2$ is given by 
\begin{equation*}
    \AW_p(\X^1,\X^2) \coloneqq \inf_{\pi \in \cplbc(\X^1,\X^2)} \E_\pi[d(X^1,X^2)^p]^\frac1p.
\end{equation*}

Replacing in the above definition $d$ by some bounded but topologically equivalent metric, \emph{e.g.} by $\tilde{d}\coloneqq d \wedge 1$, induces an equivalence relation on $\FPclass$ given by $\X^1 \sim \X^2 \iff \Ac\Wc_{\tilde d}(\X^1,\X^2)=0$.
Following \cite{BaBePa21}, we write $\FP$ for the quotient space with respect to this equivalence relation and $\FP_p \coloneqq \{ \X \in \FP \,\vert\, \mathbb E[d(X,\hat x)^p] < \infty \text{ for some } \hat x \in \Xc \}$.
We remark that by \cite{BaBePa21} the metric space $(\FP_p,\AW_p)$ is complete and separable, thus Polish, and we will from now on identify elements of $\FPclass$ with their equivalence class in $\FP$. We endow ${\rm FP}$ with the topology generated by the metric $\Ac\Wc_{\tilde d}$ and shall call it the \emph{adapted weak topology}.

\medskip
Given a measurable cost function $c : \Xc \times \Xc \longrightarrow \R$, we denote by \begin{equation} \label{eqn:def_AW_CW}
\Cc\Wc_c(\X^1,\X^2) \coloneqq \inf_{\pi \in \cplc(\X^1,\X^2)} \E_\pi [c(X^1,X^2)]\quad  {\rm and}\quad \AW_c(\X^1,\X^2) \coloneqq \inf_{\pi \in \cplbc(\X^1,\X^2)} \E_\pi [c(X^1,X^2)]
\end{equation} 
the optimal causal, resp.\ bicausal, transport value between $\X^1$ and $\X^2$ with respect to $c.$

\section{Multicausal transport} \label{sec:multicausal}
In this section, we study the multicausal transport problem, its attainment, and duality. Let $c: \Xc^N \longrightarrow \R $ be measurable and $(\X^1,\ldots,\X^N) \in {\rm FP}^N$. 
We consider the following multimarginal transport problem over multicausal couplings:
\begin{equation}
    \label{eq:def.mcot}
    V_c^{\rm mc}(\X^1,\ldots,\X^N) \coloneqq \inf_{ \pi \in \cplmc(\X^1,\ldots,\X^N)} \E_\pi\big[c(X^1,\ldots,X^N)\big].
\end{equation}

\begin{remark} \label{rem:generalXi} More generally, one can consider processes $\X^1,\ldots,\X^N$, where $X^i_t \in \Xc_t^i$ for every $t \in \{1,\ldots,T\},\; i \in \{1,\ldots,N\}$, and $\Xc_t^i$, $i \in \{1,\ldots,N\}$, are potentially different Polish spaces. The results of this section as well as {\rm\Cref{sec:CWbary,sec:AWbary}} remain true, \emph{mutatis mutandis}.
\end{remark}

The multicausal transport problem can be seen as a generalization of the classical multimarginal optimal transport problem to the
dynamic setting.
Here, in order to properly treat filtrations, the multicausality constraint is imposed on the set of multimarginal couplings.
We remark that, in the setting of \cite[Section 3.1]{KrPa23}, the duality for multicausal transport precisely corresponds to a robust pricing-hedging duality, where multicausal couplings take the role of pricing measures and admissible dual potentials can be seen as robust hedging strategies.

\subsection{Existence of primal optimizers}
First we treat existence of solutions of \eqref{eq:def.mcot}. \Cref{thm:multicausal} shows that under the usual assumptions for optimal transport, the multicausal problem \eqref{eq:def.mcot} admits optimizers, is lower-semicontinuous with respect to the marginals, and enjoys the dynamic programming principle. Such a result was already noted in \citeauthor*{BaBeLiZa17} \cite{BaBeLiZa17} and \citeauthor*{PfPi14} \cite{PfPi14} in the bimarginal case.

\begin{theorem} \label{thm:multicausal}
    Let $c$ be lower-semicontinuous and lower-bounded.
    The value map $V_c^{\rm mc}$ satisfies:
    \begin{enumerate}[label = (\roman*)]
        \item \label{it:multicausal.dpp} Dynamic programming principle: set $V(T,\overline{\omega}) \coloneqq c(X^1(\omega^1),\ldots, X^N(\omega^N))$ and define inductively backwards in time, for $t \in\{1,\ldots,T-1\}$, the value process
        \begin{align*}
        V(t,\overline{\omega}_{1:t}) &\coloneqq \inf \bigg\{ \int V(t+1,\overline{\omega}_{t+1}) \pi(\mathrm{d}\overline{\omega}_{t+1}) \,\bigg\vert\, \pi \in \cpl\big( \P^1_{t+1,\omega^\smallfont{1}_{\smallfont{1}\smallfont{:}\smallfont{t}}},\ldots,\P^N_{t+1,\omega^\smallfont{N}_{\smallfont{1}\smallfont{:}\smallfont{t}}}\big)\bigg\}, \\
        V(0) &\coloneqq \inf \bigg\{ \int V(1,\overline{\omega}_{1}) \pi(\mathrm{d}\overline{\omega}_{1}) \,\bigg\vert\, \pi \in \cpl\big( \P^1_{1},\ldots,\P^N_{1}\big)\bigg\}.
        \end{align*}
        Then $V_c^{\rm mc}(\X^1,\ldots,\X^N)$ agrees with the value $V(0).$
        
        \item \label{it:multicausal.minimizer} A coupling $\pi^\star \in \cplmc(\X^1,\ldots,\X^N)$ is optimal for $V_c^{\rm mc}(\X^1,\ldots,\X^N)$ if and only if there exists a version of the disintegration $\pi^\star(\mathrm{d}\overline{\omega})=K_1^\star(\mathrm{d}\overline{\omega}_1) \otimes K_2^\star(\overline{\omega}_1;\mathrm{d}\overline{\omega}_2) \otimes \cdots \otimes K_T^\star(\overline{\omega}_{1:T-1};\mathrm{d}\overline{\omega}_T)$ such that \begin{align*}
        K_t^\star(\overline{\omega}_{t-1};\,\cdot\,) &\in \arg\min \bigg\{ \int V(t+1,\overline{\omega}_{t+1}) \pi(\mathrm{d}\overline{\omega}_{t+1}) \,\bigg\vert\, \pi \in \cpl\big( \P^1_{t+1,\omega^\smallfont{1}_{\smallfont{1}\smallfont{:}\smallfont{t}}},\ldots,\P^N_{t+1,\omega^\smallfont{N}_{\smallfont{1}\smallfont{:}\smallfont{t}}}\big)\bigg\},\;t \in \{2,\ldots,T\}, \\
        K_1^\star(\,\cdot\,) &\in \arg\min \bigg\{ \int V(1,\overline{\omega}_{1}) \pi(\mathrm{d}\overline{\omega}_{1}) \,\bigg\vert\, \pi \in \cpl\big( \P^1_{1},\ldots,\P^N_{1}\big)\bigg\}.
        \end{align*}
        
        \item \label{it:multicausal.existence} There exists $\pi^\star \in \cplmc(\X^1,\ldots,\X^N)$ minimizing \eqref{eq:def.mcot}.
      \item \label{it:multicausal.valuelsc} The optimal value map $V_c^{\rm mc}: \FP^N\longrightarrow \R \cup \{\infty\}$ is lower-semicontinuous, where we endow $\FP^N$ with the product of adapted weak topologies on $\FP$.
    \end{enumerate}
\end{theorem}

\begin{remark} \label{rem:generalization_V_c} $(1)$ The lower-boundedness assumption on the cost function $c$ can be relaxed by standard techniques to the situation where there exist $\ell^i \in \L^1({\rm Law}_{\P^i}(X^i))$, $i \in \{1,\ldots,N\}$, such that
$\sum_{i=1}^N \ell^i(x^i) \leq c(x^1,\ldots, x^N)$, where ${\rm Law}_{\P^i}(X^i) \in \Pc(\Xc)$ denotes the law of $X^i$ under $\P^i.$

\medskip
$(2)$ If $c$ is continuous and, for $i \in \{1,\ldots,N\},$ $A^i \subseteq {\rm FP}$ is such that there exist $\varepsilon>0$ and $\ell^i :\Xc \longrightarrow \R$ with
\[
\sup_{\X \in A^i} \E_{\P^\X} \lvert \ell^i(X) \rvert^{1+\varepsilon}<\infty\quad \text{for all } i \in \{1,\ldots,N\},
\] 
and
$\lvert c(x^1,\ldots, x^N) \rvert \leq \sum_{i=1}^N \ell^i(x^i)$, then the optimal value map $V_c^{\rm mc}: \prod_{i=1}^N A^i\longrightarrow \R$ is continuous. This can be shown analogously to {\rm\Cref{thm:multicausal}}.$\ref{it:multicausal.valuelsc}$, using the fact that the map
\[ A_C\ni\mu \longmapsto \int c(x^1,\ldots,x^N) \mu(\mathrm{d} x^1,\ldots,\mathrm{d}x^N)\] is continuous for every $C>0$, where 
\[ A_C\coloneqq \bigg\{\mu \in \Pc(\Xc^N)\,\bigg\vert\, \int \big|c(x^1,\ldots,x^N)\big|^{1+\varepsilon} \mu(\mathrm{d}x^1,\ldots,\mathrm{d}x^N) < C \bigg\},\]
by application of {\rm \citeauthor*{FeKaLi20} \cite[Corollary 2.8]{FeKaLi20}.
}
\end{remark}

\begin{proof}[Proof of Theorem~\ref{thm:multicausal}] 
By \Cref{ass:continuity}, we have that $\omega^i_{1:t} \longmapsto \P^i_{t+1,\omega^\smallfont{i}_{\smallfont{1}\smallfont{:}\smallfont{t}}}$ is a continuous map from $\Omega^i_{1:t}$ to $\Pc(\Omega^i_{t+1})$ for any $i \in \{1,\ldots N\}$ and $t \in \{1,\ldots,T-1\}$.
    Moreover, by standard arguments, we can show that, for $t \in \{1,\ldots, T \}$ and $\overline{\omega}_{1:t} \in \overline{\Omega}_{1:t}$, the sets 
    \[
        \cpl\big( \P^1_{t+1,\omega^\smallfont{1}_{\smallfont{1}\smallfont{:}\smallfont{t}}},\ldots,\P^N_{t+1,\omega^\smallfont{N}_{\smallfont{1}\smallfont{:}\smallfont{t}}}\big) \quad {\rm and} \quad \cpl\big( \P^1_{1},\ldots,\P^N_{1}\big)
    \] 
    are compact, which remains true under the refined topology introduced in \Cref{lem:topology}. Thus, since $c$ is lower-bounded and lower-semicontinuous, we can inductively backwards in time verify that the map $\overline{\omega}_{1:t} \longmapsto V(t,\overline{\omega}_{1:t})$ is lower-bounded and lower-semicontinuous for every $t \in \{1,\ldots,T\}$. Further, using the measurable selection theorem \cite[Corollary 1]{MeasSel}, we conclude that for any $t \in \{0,\ldots,T-1\}$ there exist measurable minimizers for each problem $V(t,\overline{\omega}_{1:t})$, $t \in \{1,\ldots,T-1\},$ and $V(0)$.
    
    \medskip
    Invoking Lemma \ref{disintegr} we have that a measure $\pi \in \Pc(\overline{\Omega})$ belongs to $\cplmc(\X^1,\ldots,\X^N)$ if and only if there exist kernels $(K_t)_{t = 1}^T$ where $K_t$ is $\overline{\Fc}_{t - 1}$-measurable and such that \eqref{eq:def_via_disintgr} holds.
    Hence, \eqref{eq:def.mcot} can be rewritten as the minimization of
    \[
        \int c\big(X^1\big(\omega^1\big),\ldots,X^N\big(\omega^N\big)\big) \mathrm{d} (K_1 \otimes \ldots \otimes K_T)(\mathrm{d} \overline{\omega}),
    \]
     over all such kernels $(K_t)_{t = 1}^T$. Thus, for any $ \pi \in \cplmc(\X^1,\ldots,\X^N)$, we clearly have \[ \int c\big(X^1\big(\omega^1\big),\ldots,X^N\big(\omega^N\big)\big) \pi(\mathrm{d}\overline{\omega})=\int c\big(X^1\big(\omega^1\big),\ldots,X^N\big(\omega^N\big)\big) \mathrm{d} (K_1 \otimes \ldots \otimes K_T)(\mathrm{d} \overline{\omega})\geq V(0). \] 
     Moreover, the inequality becomes an equality if and only if there exist kernels $K^\star_t : \overline{\Omega}_{1:t-1}\longrightarrow \Pc(\overline{\Omega}_{t}), t\in\{2,\ldots,T\}$,  and $K^\star_1\in\Pc(\overline{\Omega}_{1})$,
     such that     
     \[
        \pi^\star(\mathrm{d} \overline{\omega}) =K_1^\star(\mathrm{d}\overline{\omega}_1) \otimes\cdots \otimes K_T^\star(\overline{\omega}_{1:T-1}; \mathrm{d}\overline{\omega}_T),
    \]  and \begin{align*}
        K_t^\star(\overline{\omega}_{t-1};\,\cdot\,) &\in \arg\min \bigg\{ \int V(t+1,\overline{\omega}_{t+1}) \pi(\mathrm{d}\overline{\omega}_{t+1}) \,\bigg\vert\, \pi \in \cpl\big( \P^1_{t+1,\omega^\smallfont{1}_{\smallfont{1}\smallfont{:}\smallfont{t}}},\ldots,\P^N_{t+1,\omega^\smallfont{N}_{\smallfont{1}\smallfont{:}\smallfont{t}}}\big)\bigg\},\;t \in \{2,\ldots,T\}, \\
        K_1^\star(\,\cdot\,) &\in \arg\min \bigg\{ \int V(1,\overline{\omega}_{1}) \pi(\mathrm{d}\overline{\omega}_{1}) \,\bigg\vert\, \pi \in \cpl\big( \P^1_{1},\ldots,\P^N_{1}\big)\bigg\}.
        \end{align*} This shows $\ref{it:multicausal.dpp}$, $\ref{it:multicausal.minimizer}$ and $\ref{it:multicausal.existence}$. 
        
        \medskip As for item~$\ref{it:multicausal.valuelsc}$, let $(\X^{1,n},\ldots,\X^{N,n}) \in {\rm FP}^N,$ $n \in \N,$ be a sequence such that
        \[(\X^{1,n},\ldots,\X^{N,n}) \longrightarrow (\X^{1,\infty},\ldots,\X^{N,\infty})\quad \text{in } {\rm FP}^N.\]
        Let, for $n \in \N,$ $\pi^n \in \cplmc(\X^{1,n},\ldots,\X^{N,n})$ be an optimizer for the problem $V_c^{\rm mc}(\X^{N,n},\ldots,\X^{N,n}),$ which exists thanks to $\ref{it:multicausal.existence}$. Using \Cref{lem:multic_ip}.$\ref{itm:multic_ip1}$, we have that \[\gamma^n \coloneqq {\rm  Law}_{\pi^n}({\rm ip} (\X^{1,n}),\ldots,{\rm ip}(\X^{N,n})) \in \cplmc({\rm ip} (\X^{1,n}),\ldots,{\rm ip}(\X^{N,n})), \quad n \in \N.\] 
        By \Cref{lem:comp_multic}, up to passing to a subsequence,  there is a weak limit $\gamma^\infty \in \cplmc({\rm ip} (\X^{1,\infty}),\ldots,{\rm ip}(\X^{N,\infty}))$ of the sequence $(\gamma^n)_{n \in \N}$, which moreover  satisfies $\gamma^\infty={\rm Law}_{\pi^{\infty}}({\rm ip} (\X^{1,\infty}),\ldots,{\rm ip}(\X^{N,\infty}))$ for some $\pi^\infty \in \cplmc(\X^{1,\infty},\ldots,\X^{N,\infty})$ by \Cref{lem:multic_ip}.$\ref{itm:multic_ip2}$.
        It then follows that
        \begin{align} \label{eqn:lsc}
        \begin{split}
        V_c^{\rm mc}(\X^{1,\infty},\ldots,\X^{N,\infty}) &\leq \int c(X^{1,\infty}(\omega^{1}),\ldots,X^{N,\infty}(\omega^N)) \pi^{\infty}(\mathrm{d}\omega^1,\ldots,\mathrm{d}\omega^N)\\
        &=\int c(x^1,\ldots,x^N) \gamma^{\infty}(\mathrm{d}(x^1,p^1),\ldots,\mathrm{d}(x^N,p^N)) \\
        &\leq \liminf_{n \rightarrow \infty} \int c(x^1,\ldots,x^N) \gamma^{n}(\mathrm{d}(x^1,p^1),\ldots,\mathrm{d}(x^N,p^N)) \\
        &=\liminf_{n \rightarrow \infty} \int c(X^{1,n}(\omega^{1}),\ldots,X^{N,n}(\omega^N)) \pi^{n}(\mathrm{d}\omega^1,\ldots,\mathrm{d}\omega^N)\\
        &=\liminf_{n \rightarrow \infty} V_c^{\rm mc}(\X^{N,n},\ldots,\X^{N,n}),
        \end{split}
        \end{align} where we have used lower-semicontinuity and lower-boundedness of $c$ in the third line. This concludes the proof.
\end{proof}

\begin{remark} Existence of a solution can be achieved by standard arguments using compactness and lower-semicontinuity. Indeed, {\rm\Cref{lem:compactness}} gives compactness of the set of admissible couplings, which in turn together with lower-boundedness and lower-semicontinuity of $c$ yields existence of a minimizer. 
\end{remark}

\subsection{Duality}
In this section, we study the dual representation of the multicausal transport problem. For this, we need to introduce the following set of functions, which will be used to test multicausality of couplings, see \Cref{testfunc}. We set $\mathfrak{F}^{{\rm mc}}\coloneqq \bigoplus_{i=1}^N \mathfrak{F}^{i,{\rm mc}}$, where, for $i \in \{1,\ldots,N\}$,
\begin{multline*}
\mathfrak{F}^{i,{\rm mc}}\coloneqq \bigg\lbrace
F^i: \overline{\Omega} \longrightarrow \R\,\bigg\vert\, F^i(\overline{\omega})=\sum_{t=2}^{T} \bigg[ a^i_t \big((\omega^j_{1:t-1})_{j \neq i},\omega^i_{1:t}\big)-\int a_t^i((\omega^j_{1:t-1})_{j \neq i},\omega^i_{1:t-1},\tilde{\omega}^i_t)\,\P^i_{t,\omega^i_{\smallfont{1}\smallfont{:}\smallfont{t}\smallfont{-}\smallfont{1}}}(\mathrm{d}\tilde{\omega}^i_{t} )  \bigg], \\
a_t^i \in \bigcap_{\pi \in \cplmc(\X^1,\ldots,\X^N)}\L^1(\pi),\; t \in \{ 2,\ldots,T \} 
\bigg\rbrace.
\end{multline*}

Note that each function $F \in \mathfrak{F}^{{\rm mc}}$  is associated with a process $(F_t)_{t = 2}^T$ where $F_T = F$ and for $t \in \{2,\ldots,T$\},
\[ F_t(\overline{\omega}) \coloneqq \sum_{s=2}^{t}\sum_{i=1}^N \bigg[ a^i_s \big((\omega^j_{1:s-1})_{j \neq i},\omega^i_{1:s}\big)-\int a_s^i((\omega^j_{1:s-1})_{j \neq i},\omega^i_{1:s-1},\tilde{\omega}^i_s)\,\P^i_{s,\omega^i_{\smallfont{1}\smallfont{:}\smallfont{s}\smallfont{-}\smallfont{1}}}(\mathrm{d}\tilde{\omega}^i_{s} )  \bigg], \]
which is a $\pi$-martingale for any $\pi \in \cplmc(\X^1,\ldots,\X^N).$

\begin{lemma} \label{testfunc} Let $\pi \in {\rm cpl}(\X^1,\ldots,\X^N).$ Then $\pi \in {\rm cpl}_{\rm mc}(\X^1,\ldots,\X^N)$ if and only if
\[\sup_{F \in \mathfrak{F}^{{\rm mc}}}\int F \mathrm{d}\pi=0 \;\text{or, equivalently, } \sup_{F \in \mathfrak{F}^{{\rm mc}}}\int F \mathrm{d}\pi< \infty.\]
\end{lemma}
\begin{proof} The statement follows from standard arguments, see \citeauthor*{BaBeLiZa17} \cite[Proposition 2.4]{BaBeLiZa17}.
\end{proof}

\medskip
Let us consider the following problem
\begin{multline*}
 D_c^{\rm mc}\coloneqq\sup \bigg\{ \sum_{i=1}^N \int f^i(\omega^i) \P^i(\mathrm{d}\omega^i) \,\bigg\vert\,  (f^i)_{i=1}^N \in \prod_{i=1}^N \L^1(\P^i): \\ \exists F \in \mathfrak{F}^{{\rm mc}} \text{ such that } \bigoplus_{i=1}^N f^i \leq c(X^1,\ldots, X^N)+F \bigg\}.
 \end{multline*} 

\begin{remark} We immediately obtain the weak duality result $D_c^{{\rm mc}} \leq V_c^{\rm mc}.$ Indeed, we have
\[  \sum_{i=1}^N \int f^i \mathrm{d}\P^i= \int \Big(\bigoplus_{i=1}^N f^i \Big)\mathrm{d}\pi \leq \int (c+F) \mathrm{d}\pi = \int c \mathrm{d}\pi, \] for any $\pi \in \cplmc(\X^1,\ldots,\X^N)$ and $f^i \in \L^1(\P^i)$ such that the inequality $\bigoplus_{i=1}^N f^i \leq c(X^1,\ldots,X^N) +F$ holds $\cplmc(\X^1,\ldots,\X^N)$--quasi-surely for some $F \in \mathfrak{F}^{{\rm mc}}.$
\end{remark}

The main theorem of this section follows.

\begin{theorem} \label{thm:dual_multicausal} Let $c :\Xc^{N} \longrightarrow \R$ be measurable and assume that there exist $\ell^i \in \L^1({\rm Law}_{\P^i}(X^i)),$ $i\in \{1,\ldots,N\}$, such that  $c \leq \bigoplus_{i=1}^N \ell^i.$ Then $D_c^{{\rm mc}}=V_c^{\rm mc}.$ Moreover, if either side is finite, then $D_c^{{\rm mc}}$ is attained.
\end{theorem}
\begin{proof}
    This result is covered by \cite[Theorem 4.13]{KrPa23}.
\end{proof}

\section{Bicausal barycenters} \label{sec:AWbary}
This section is concerned with the problem of barycenters in bicausal setting. To have an application in mind,
suppose there are $N$ experts whose different believes are summarized by the filtered processes $\X^1,\dots,\X^N$. 
Since we trust each of these experts, we want to merge their believes and find a consensus model $\Y$.
Therefore, we want to find a filtered process that is in a suitable sense close to the experts' beliefs.
The latter can be adequately formalized by working with adapted Wasserstein distances, which thus naturally leads to the following general setting.

\medskip Let $\Yc=\prod_{t=1}^T\Yc_{t}$ be a product of Polish spaces $\Yc_t,\; t \in \{1,\ldots,T\},$ endowed with the metric $d_{\Yc}(y,\tilde y)\coloneqq \sum_{t=1}^T d_{\Yc_t}(y_t,\tilde y_t),$ for $(y, \tilde y)=(y_{1:T},\tilde y_{1:Y}) \in \Yc^2.$
Let further $c^i : \Xc \times \Yc \longrightarrow \R$, $i \in \{1,\ldots,N \}$, be given measurable functions. For fixed filtered processes $\X^1,\ldots, \X^N$, we proceed to study the barycenter problem 
\begin{equation} \label{eq.bc_bary}
B_{c^{1:N}}^{\rm bc}\coloneqq\inf_{\Y\in {\rm FP}(\Yc)} B_{c^{1:N}}^{\rm bc}(\Y), \end{equation} 
where
\[ B_{c^{1:N}}^{\rm bc}(\Y)\coloneqq \sum_{i=1}^N \Ac\Wc_{c^i}(\X^i,\Y), \quad \Y \in {\rm FP}(\Yc), \]
and where $\Ac\Wc_{c^i}$ is defined analogously to \eqref{eqn:def_AW_CW}. Here, ${\rm FP}(\Yc)$ as before denotes the factor space of all filtered processes 
\[
    \Y = \Big( \Omega^\Y, \Fc^\Y,  \F^\Y, \P^\Y, Y \Big)
\] 
such that $Y_t$ takes values in $\Yc_t$ for every $t \in \{1,\ldots,T\}.$ 
For $p \ge 1$ we shall similarly as before denote by ${\rm FP}_p(\Yc)$ the subset of filtered processes with finite $p$-th moments. 

\begin{remark} Our setting clearly includes the case $\Yc = \Xc$, which is the natural space for seeking barycenters. However, for clarity of presentation and to keep full generality, we consider a more general space $\Yc$, as certain assumptions will be imposed on $\Yc$ but are not required for $\Xc$.
\end{remark}

\begin{remark}
    As for the barycenter problem in classical optimal transport, finding explicit solutions for the barycenter problem \eqref{eq.bc_bary} is in general a non-trivial task.
    In the classical case there are two prominent instances where the barycenter is explicitly known, for $c^i(x,y) = \|x - y\|^p$, $x,y \in \R^d$.
    The first one is the case when $N = 2$, as in this case the barycenter can be constructed via displacement interpolation of the $\mathcal W_p$-optimal coupling.
    The second case is when $N$ is arbitrary, $p = 2$, $\|\cdot\|$ is the Euclidean norm, and the marginals consist only of Gaussian measures.
    Then the barycenter is again Gaussian with covariance matrix implicitly given as the solution to a fixed-point equation.

    \medskip For the bicausal barycenter problem, when $N = 2$ and $c^i(x,y) = \|x - y\|^p$, $x,y \in \R^{d T}$, one similarly has that the barycenter can be constructed via displacement interpolation of the $\mathcal{AW}_p$-optimal coupling.
    For more details see {\rm\cite[Section 5.4]{BaBePa18}}.
    Recently, the bicausal transport problem between Gaussians was studied in {\rm\cite{GuWo24, AcHoPa24}}, while the study of the bicausal barycenter problem between Gaussians is ongoing research. 
\end{remark}

\subsection{Existence of primal optimizers}
Let us first deal with existence of solutions to the barycenter problem. One can easily see that in general there is no reason for a solution to exist. In order to obtain existence of solutions, we will assume the following conditions.

\begin{assumption} \label{assAdaptBary} 
We assume that: 
\begin{enumerate} [label = (\roman*)]
    \item \label{assAdaptItem1} the functions $c^i$ are lower-semicontinuous and lower-bounded for every $i \in \{1, \ldots, N \};$
    \item \label{assAdaptItem3} there exists $y^0 \in \Yc$ such that the set $ \{ y \in \Yc \,\vert\, d_\Yc(y^0,y)\leq K \}$ is compact for every $K \geq 0.$ Moreover,  the deterministic process $\Y^0 \in {\rm FP}(\Yc)$ such that $Y^0\equiv y^0$ satisfies $B_{c^{1:N}}^{\rm bc}(\Y^0)<\infty;$
  \item \label{assAdaptItem2} there exist $C\geq 0$ and $\ell^1 \in \L^1({\rm Law}_{\P^1}(X^1))$ such that $d_\Yc(y^0,y) \leq C\big(1+c^1(x,y)+\ell^1(x)\big)$ for all $(x,y) \in \Xc \times \Yc$.
\end{enumerate}
\end{assumption} 
\begin{remark} $(1)$ Under {\rm\Cref{assAdaptBary}}, we have in particular that $-\infty < B_{c^{1:N}}^{\rm bc} < \infty.$

\medskip $(2)$ Condition $\ref{assAdaptItem1}$ is a standard assumption in optimal transport theory and the lower-boundedness assumption can be slightly relaxed to $c^i(x,y)\geq \ell^i(x)$ for some $\ell^i \in \L^1({\rm Law}_{\P^i}(X^i)).$

\medskip
$(3)$ Conditions $\ref{assAdaptItem2}$ and $\ref{assAdaptItem3}$ ensure that we can restrict the set of admissible filtered processes and use compactness arguments to obtain existence.

\medskip $(4)$ If $\Yc$ is a Banach space, then condition $\ref{assAdaptItem3}$ is satisfied if and only if $\Yc$ is finite-dimensional.
\end{remark}

\begin{theorem} \label{AdaptBaryExist} Under {\rm\Cref{assAdaptBary}}, there exists $\bar{\Y} \in {\rm FP}_1(\Yc)$ such that $B_{c^{1:N}}^{\rm bc}=B_{c^{1:N}}^{\rm bc}(\bar{\Y}).$
\end{theorem}
\begin{proof}
The proof follows analogous steps to those in the proof of \citeauthor*{BaBePa21} \cite[Theorem 6.7]{BaBePa21}. We obtain from \Cref{thm:multicausal} that the map $\Y \longmapsto B_{c^{1:N}}^{\rm bc}(\Y)$ is lower-semicontinuous with respect to the adapted weak topology. Hence, it is sufficient to show that $B_{c^{1:N}}^{\rm bc}=\inf_{\Y\in \Kc} B_{c^{1:N}}^{\rm bc}(\Y)$ for some set $\Kc \subseteq {\rm FP}(\Yc)$ that is relatively compact in the adapted weak topology. To this end, we assume without loss of generality that $c^i$ is non-negative for every $i \in \{1,\ldots,N\}$ and set 
\[ \Kc\coloneqq\{ \Y \in {\rm FP}(\Yc) : B_{c^{1:N}}^{\rm bc}(\Y) \leq B_{c^{1:N}}^{\rm bc} + 1  \}.\] 
Then, for every $\Y \in \Kc$, we have
\begin{align} \label{eqn:baryproof}
\begin{split}\E_{\P^\Y} [d_\Yc(y^0,y)] = \Ac\Wc_{d_{\Yc}}(\Y,\Y^0) &\leq C \big(1+ \Ac\Wc_{c^1}(\X^1,\Y) + \E_{\P^1}[\ell^1(X^1)]\big)\\
&\leq  C \big(1+B_{c^{1:N}}^{\rm bc}(\Y) + \E_{\P^1}[\ell^1(X^1)] \big) \leq C \big(2+B_{c^{1:N}}^{\rm bc} + \E_{\P^1}[\ell^1(X^1)] \big).
\end{split}
\end{align}
It follows that the set $\{ {\rm Law}_{\P^\Y}(Y) : \Y \in \Kc  \}$ is tight due to \Cref{assAdaptBary}.$\ref{assAdaptItem3}$, and hence $\Kc$ is relatively compact with respect to the adapted weak topology; see \cite[Theorem 5.1]{BaBePa21}. Thus, by standard results, there exists a process $\bar{\Y} \in {\rm FP}(\Yc)$ solving $B_{c^{1:N}}^{\rm bc}.$ The fact that $\bar{\Y} \in {\rm FP}_1$ 
follows immediately from \eqref{eqn:baryproof}. 
\end{proof}

\begin{remark} \label{AdaptBaryRemark} If we further have that $d^p_{\Yc}(y^0,y) \leq C(1+c^1(x,y)+\ell^1(x))$, $(x,y) \in \Xc\times \Yc$, for some $C \geq 0,$ $\ell^1 \in \L^1({\rm Law}_{\P^1}(X^1)) $ and $p \geq 1$, we can analogously as in \eqref{eqn:baryproof} verify that $\bar{\Y}$ in 
{\rm Theorem~\ref{AdaptBaryExist}} necessarily satisfies
$\bar{\Y} \in {\rm FP}_p(\Yc)$. 
\end{remark}

\begin{remark} In fact, one can consider a slightly more general setting. Let $\mu \in \Pc({\rm FP}_1)$ and let $c : \Xc \times \Yc \longrightarrow \R$ be measurable. If {\rm\Cref{assAdaptBary}} holds with $c^1$ replaced by $c$ and $B_{c^{1:N}}^{\rm bc}(\Y^0)<\infty$ replaced by $\int \Ac \Wc_c(\Z,\Y^0) \mu(\mathrm{d}\Z)<\infty$, one can analogously show that there exists a process $\bar{\X} \in {\rm FP}_1$ solving
\[ \inf_{\Y \in {\rm FP}(\Yc)} \int \Ac \Wc_c(\Z,\Y) \mu(\mathrm{d}\Z).\]
A similar setting was also considered by {\rm \citeauthor*{BaBePa21} \cite[Theorem 6.7]{BaBePa21}}. 
\end{remark} 

\begin{remark} Let us emphasize that even in the case where the filtration $\mathcal{F}^i$ is canonical in the sense that it is generated by the process $X^i$ for every $i \in \{1, \ldots, N\}$, there is no guarantee that there exists a barycenter $\bar \X$ with its filtration generated by the corresponding process $\bar X$; see also {\rm\Cref{subsec:bary_multi}}, in particular {\rm \Cref{rem:filtr_bary}}.
\end{remark}

\begin{example} \label{rem:baryeucnorm} An important example is the case $\Yc_t=\Xc_t=\R^d$, for some $d \in \N$, and $c^i(x,y)=\lambda^i \lVert x-y \rVert_p^p$, where $p \geq 1$ and $\lambda^i$ are positive weights satisfying $\sum_{i=1}^N \lambda^i=1.$ It is straightforward to verify that if $\X^i \in {\rm FP}_p$ for every $i \in \{1,\ldots,N\}$, then {\rm\Cref{assAdaptBary}} is satisfied with $y^0=0$ and $\ell(x)=\|x\|_p^p.$ Hence, according to {\rm\Cref{AdaptBaryExist}} and {\rm\Cref{AdaptBaryRemark}}, there exists a process $\bar{\Y} \in {\rm FP}_p$ solving \[ \inf_{\Y \in {\rm FP}}\sum_{i=1}^N \lambda^i\Ac\Wc_p^p(\X^i,\Y). \]
\end{example}

\subsection{Bicausal barycenters and multicausal OT} \label{subsec:bary_multi}
We proceed to the characterization of the barycenter problem by means of a multimarginal reformulation in the bicausal setting.

\begin{remark} The results presented in this work can be analogously derived for the barycenter problem in classical optimal transport. One thus obtains slightly more general results than the ones presented  by {\rm \citeauthor*{AgCa11} \cite{AgCa11}}, where the quadratic cost was considered.
\end{remark}

\medskip
In order to connect the barycenter problem to a multicausal problem, we introduce the following assumption on the structure of the functions $c^i$'s.

\begin{assumption} \label{ass.baryandmulti} For $i \in\ \{1,\ldots,N\}$, the functions $c^i : \Xc \times \Yc \longrightarrow \R$ are of the form
\[ 
c^i(x,y)=\sum_{t=1}^T c^i_t (x_t,y_t),\quad (x,y) \in \Xc \times \Yc,
\] 
where
$c_t^i : \Xc_t \times \Yc_t \longrightarrow \R$, $t \in \{1,\ldots,T\}$, are lower-bounded measurable functions. Moreover, the function $c: \Xc^N \longrightarrow \R$ defined by \begin{equation} \label{eqn:bary_multi_c}c(x^{1:N})\coloneqq \inf_{y \in \Yc} \sum_{i = 1}^N c^i(x^i, y)\end{equation} is Borel measurable.
\end{assumption}

\begin{example} \label{example:separ.cost} {\rm\Cref{ass.baryandmulti}} is satisfied \emph{e.g.} if $\Yc=\Xc,$ and $c^i(x,y)=d(x,y)=\sum_{t=1}^T d_t(x_t,y_t).$ Another example is the case when $\Yc_t=\Xc_t=\R^d$ for some $d \in \N$ and for every $t \in \{1,\ldots,T\}$, and $c^i(x,y)=\lVert x-y \rVert_p^p=\sum_{t=1}^T \lVert x_t-y_t \rVert_p^p$ for some $p\geq 1$. 
\end{example}

\begin{remark} In {\rm\Cref{ass.baryandmulti}} we can more generally consider $c^i$, $i \in \{1,\ldots,N\}$, of the form
\[ 
c^i(x,y)=\sum_{t=1}^T c^i_t (x_{1:t},y_t),\quad (x,y) \in \Xc \times \Yc,
\] for some lower-bounded measurable functions $c_t^i : \Xc_{1:t} \times \Yc_t \longrightarrow \R,\; t \in \{1,\ldots,T\}$, and the results of this section remain \emph{mutatis mutandis} valid. 
\end{remark}

\begin{remark} The function $c$ defined in {\rm\Cref{ass.baryandmulti}} is Borel measurable \emph{e.g.} if the function $\sum_{i=1}^N c^i$ is continuous or if {\rm\Cref{assAdaptBary}} holds, see {\rm Remarks~\ref{rem:barymultiborel}} and {\rm\ref{rem:phi0nonepmty}}.
\end{remark}
For $c$ defined in \Cref{ass.baryandmulti}, let us recall that we denote by $V_{c}^{\rm mc}$ the following multicausal OT problem: 
\begin{equation}\label{eq.Vmci}
    V_{c}^{\rm mc}=\inf_{\pi \in \cplmc(\X^1,\ldots,\X^N)} \E_\pi \big[c (X^1,\ldots,X^N)\big].
\end{equation}
The aim of this section is to link  problem $V_{c}^{\rm mc}$ to the bicausal barycenter problem.

\begin{remark} \label{rem:univmeas}
Under {\rm \Cref{ass.baryandmulti}}, there exists for every $t \in \{1,\ldots, T\}$ and $\varepsilon>0$ a universally measurable function $\phi^\varepsilon_t : \Xc_t^N \longrightarrow \Yc_t$ satisfying, for any $x_t^{1:N} \in \Xc_t^N$,  
\begin{align}\label{eqn:meas.minim} \sum_{i=1}^N c^i_t (x_t^i,\phi^\varepsilon_t(x_t^{1:N})) \leq\inf_{y_t \in \Yc_t} \sum_{i=1}^N c^i_t (x^i_t,y_t)+\varepsilon;
\end{align} 
see {\rm\cite[Theorem 2]{MeasSel}}. Moreover, for any $\pi \in \cplmc(\X^1,\ldots,\X^N)$, thanks to {\rm\cite[Lemma 1.2]{crauel2002}} there exists a Borel measurable map $\phi^{\varepsilon,\pi}_t$ such that $\phi^{\varepsilon,\pi}_t=\phi^{\varepsilon}_t$ holds $\pi$--almost surely.

\medskip
For any $\pi \in \cplmc(\X^1,\ldots,\X^N)$, we denote by $\phi^{\varepsilon,\pi}$ the map $\phi^{\varepsilon,\pi}\coloneqq(\phi^{\varepsilon,\pi}_t)_{t=1}^T : \Xc^N \longrightarrow \Yc$ and define the following filtered process
\[\Y^{\pi,\phi^{\varepsilon,\pi}}\coloneqq\left(\overline{\Omega},\overline{\Fc}_T,\overline{\F}, \pi, Y^{\phi^{\varepsilon,\pi}} \right),\]
where $Y^{\phi^{\varepsilon,\pi}}$ is given by $Y^{\phi^{\varepsilon,\pi}}\coloneqq(\phi^{\varepsilon,\pi}_t(Y^1_t,\ldots,Y^N_t))_{t=1}^T.$ Note that $Y^{\phi^{\varepsilon,\pi}}$ is $\overline{\F}$-adapted.
\end{remark}

We first show that the values of the multicausal and barycenter problems coincide. 

\begin{lemma} \label{lem:barymultilemma} Under {\rm\Cref{ass.baryandmulti}}, the multicausal problem \eqref{eq.Vmci} and the barycenter problem \eqref{eq.bc_bary} have the same value, that is, $B_{c^{1:N}}^{\rm bc}= V_{c}^{\rm mc}.$
\end{lemma}
\begin{proof} 
Let us first prove that $B_{c^{1:N}}^{\rm bc} \geq V_{c}^{\rm mc}.$ Let $\Y \in {\rm FP}(\Yc)$ and $\pi^i \in \cplbc(\X^i,\Y)$ be arbitrary. According to \Cref{lemma:glueing}, there exists $\Q \in \Pc (\overline{\Omega} \times\Omega^\Y, \overline{\Fc}_T \otimes \Fc^\Y_T )$ such that $(\mathrm{proj}^i,\mathrm{proj}^{N+1})_\#\Q=\pi^i$, $i \in \{1,\ldots,N\}$, and $\pi\coloneqq{\rm proj}^{ \{1,\ldots,N \}}_\#\Q$ belongs to $\cplmc(\X^1,\ldots,\X^N)$. Then we have
\begin{equation}\label{eqn:barymultilemma} \sum_{i=1}^N \E_{\pi^i} \big[ c^i(X^i,Y) \big] =  \E_{\Q} \bigg[\sum_{i=1}^N c^i(X^i,Y) \bigg]
\geq \E_{\Q} [c(X^1,\ldots, X^N)]=\E_{\pi} [c(X^1,\ldots, X^N)] \geq V_{c}^{\rm mc}.
\end{equation} Because $\X$ and $\pi^i$ were arbitrary, the inequality follows.

\medskip Fix $\varepsilon>0$ and $\pi \in \cplmc (\X^1,\ldots,\X^N)$, and write $\pi^i\coloneqq({\rm proj}^i,{\rm id})_\# \pi.$ Then it is straightforward to verify that $\pi^i \in \cplbc (\X^i,\Y^{\pi,\phi^{\varepsilon,\pi}}),$ where $\Y^{\pi,\phi^{\varepsilon,\pi}}$ is introduced in \Cref{rem:univmeas}. It follows that
\begin{align} \label{eqn:barymultilemma2}
\begin{split}
\sum_{i=1}^N \Ac\Wc_{c^i}(\X^i,\Y^{\pi,\phi^{\varepsilon,\pi}}) &\leq \sum_{i=1}^N \E_{\pi^i} \big[c^i(X^i,Y^{\phi^{\varepsilon,\pi}})\big] =\E_{\pi}\bigg[ \sum_{i=1}^N c^i(X^i,\phi^{\varepsilon,\pi}(X^1,\ldots,X^N)) \bigg]
\\ 
& \leq \E_{\pi}\big[ c(X^1,\ldots,X^N) \big] + T \varepsilon. 
\end{split}
\end{align}
Because $\varepsilon>0$ and $\pi$ were arbitrary, the proof is concluded.
 \end{proof}

Theorem~\ref{thm:con_of_form} below, which is the main theorem of this section, asserts that, under sufficient regularity of $c^i$, $i \in \{1,\ldots,N\}$, there exists a correspondence of solutions of the multicausal problem \eqref{eq.Vmci} and the barycenter problem \eqref{eq.bc_bary}. That is to say, any solution to either problem yields a solution to the other, and, thus, if either of the problem is attained, so is the other. Before stating this correspondence, we need the following.

\begin{definition}
We write $\Phi^0$ for the set of all functions $\phi^0=(\phi^0_t)_{t=1}^T : \Xc^N \longrightarrow \Yc,$ where, for every $t \in \{1,\ldots,T\},$ $\phi^0_t : \Xc_t^N \longrightarrow \Yc_t$ is a universally measurable function satisfying
\[ \inf_{y_t \in \Yc_t} \sum_{i=1}^N c^i_t (x^i_t,y_t)= \sum_{i=1}^N c^i_t (x^i_t,\phi^0_t(x^{1:N}_t)),\quad x^{1:N}_t \in \Xc_t^N. \] 
\end{definition}

\begin{remark} \label{rem:barymultiborel}
Similarly as before, for any $\pi \in \cplmc(\X^1,\ldots,\X^N)$ there exists a Borel measurable map $\phi^{0,\pi}_t$ such that $\phi^{0,\pi}_t=\phi^{0}_t$ holds $\pi$--almost surely and we set \[\X^{\pi,\phi^{0,\pi}}\coloneqq\left(\overline{\Omega},\overline{\Fc}_T,\overline{\F}, \pi, X^{\phi^{0,\pi}} \right),\]
where $X^{\phi^{0,\pi}}$ is given by $X^{\phi^{0,\pi}}\coloneqq(\phi^{0,\pi}_t(X^1_t,\ldots,X^N_t))_{t=1}^T.$ If $\phi^0$ itself is Borel measurable, then $c$ is Borel measurable.
\end{remark}

\begin{theorem} \label{thm:con_of_form} Under {\rm\Cref{ass.baryandmulti}}, the following statements are true.
\begin{enumerate}[label = (\roman*)]
\item \label{it:bary_multi.1} 
    Assume that $\bar{\Y}$ solves $B_{c^{1:N}}^{\rm bc}$ and let $\pi_i^\star \in \cplbc(\X^i,\bar{\Y})$, $i \in \{1,\ldots,N \}$, be optimal transport plans for $\Ac\Wc_{c^i}(\X^i,\bar{\Y})$. Let $\Q^\star$ be the probability measure from {\rm\Cref{lemma:glueing}} associated to $\pi_i^\star$, $i \in \{1,\ldots,N\}$. Then $\pi^\star\coloneqq{\proj}^{\{1,\ldots,N \}}_\# \Q^\star \in \cplmc (\X^1,\ldots,\X^N )$ solves $V_{c}^{\rm mc}.$

\item \label{it:bary_multi.2} Conversely, assume that $\pi^\star \in \cplmc(\X^1,\ldots, \X^N)$ solves $V_{c}^{\rm mc}.$ Then, for any $\phi^0 \in \Phi^0$, the process $\Y^{\pi^\star,\phi^{0,\pi^\star}}$ is a solution to $B_{c^{1:N}}^{\rm bc}$. 
\end{enumerate}
\end{theorem}
\begin{proof} Recall that by \Cref{lem:barymultilemma} we have $V_{c}^{\rm mc}=B_{c^{1:N}}^{\rm bc}.$ Let us start with the proof of $\ref{it:bary_multi.1}$. Let $\pi^\star_i, \pi^\star$ and $\Q^\star$ be as in the statement. Similarly as in \eqref{eqn:barymultilemma}, we have
\begin{align*}V_{c}^{\rm mc}&=B_{c^{1:N}}^{\rm bc} = \sum_{i=1}^N \E_{\pi_i^\star} \big[ c^i(X^i,\bar{Y}) \big] =  \E_{\Q^\star} \bigg[ \sum_{i=1}^N c^i(X^i,\bar{Y}) \bigg]\\ &\geq \E_{\Q^\star} \big[ c(X^1,\ldots, X^N) \big]=\E_{\pi^\star} \big[c(X^1,\ldots, X^N) \big] \geq V_{c}^{\rm mc},
\end{align*} which immediately yields the result.

\medskip Next, we prove $\ref{it:bary_multi.2}$. Let $\pi^\star$ be as in the statement and let $\phi^0 \in \Phi^0$ be arbitrary. As in \eqref{eqn:barymultilemma2}, we find \begin{align*} 
B_{c^{1:N}}^{\rm bc} \leq \sum_{i=1}^N \Ac\Wc_{c^i}(\X^i,\Y^{\phi^{0,\pi^\star}}) &\leq \E_{\pi^\star} \bigg[\sum_{i=1}^N c^i(X^i,\phi^{0,\pi^\star}(X^1,\ldots,X^N)) \bigg] \\
&= \E_{\pi^\star} \big[c(X^1,\ldots,X^N)\big]=V_{c}^{\rm mc}=B_{c^{1:N}}^{\rm bc}.
\end{align*} 
This concludes the proof.
\end{proof}

\begin{remark}\label{rem:filtr_bary} $(1)$ {\rm\Cref{thm:con_of_form}} in particular gives an interesting insight on the structure of bicausal barycenters. More specifically, we see that the filtration is a product of all the filtrations corresponding to the processes involved. That is, in general there is no guarantee that it is generated by the corresponding process even if the filtrations $\F^i$ are generated by the processes $X^i.$ In particular cases, however, this might still be the case, since $\Y^{\pi^\star,\phi^{0,\pi^\star}}$ is just one representative of an equivalence class.

\medskip $(2)$ Note that we obtain attainment for the problem $B_{c^{1:N}}^{\rm bc}$ using {\rm\Cref{thm:con_of_form}} under different conditions than in {\rm\Cref{AdaptBaryExist}}. On the other hand, we obtain attainment of the multicausal optimal transport problem  if the cost function enjoys the specific  structure described in {\rm \Cref{ass.baryandmulti}}. 
\end{remark}

Using an appropriate measurable selection theorem, see \Cref{rem:phi0nonepmty}, we can deduce that the set $\Phi^0$ is not empty under \Cref{assAdaptBary}. In particular, we obtain existence of a solution to the corresponding multimarginal problem. 

\begin{remark} \label{rem:phi0nonepmty}
Under {\rm\Cref{assAdaptBary}} and {\rm\Cref{ass.baryandmulti}}, using the measurable selection theorem {\rm\cite[Corollary 1]{MeasSel}}, we have that there exist Borel measurable functions $\phi^0_t : \Xc_t^N \longrightarrow \Yc_t$, $t \in \{1,\ldots,T\}$, satisfying
\[ \inf_{y_t \in \Yc_t} \sum_{i=1}^N c^i_t (x^i_t,y_t)= \sum_{i=1}^N c^i_t \big(x^i_t,\phi^0_t\big(x^{1:N}_t\big)\big),\quad x^{1:N}_t \in \Xc_t^N. \]
\end{remark}

\begin{corollary} If {\rm Assumptions \ref{assAdaptBary}} and {\rm\ref{ass.baryandmulti}} hold, then $V_{c}^{\rm mc}=B_{c^{1:N}}^{\rm bc},$ and both problems admit a solution. \end{corollary}
\begin{proof} The statement is a direct consequence of \Cref{lem:barymultilemma} and \Cref{thm:con_of_form,AdaptBaryExist}.
\end{proof}

In fact, verifying that the assumptions mentioned above are satisfied has to be done on a case-by-case basis. Let us conclude this section by providing an example with quadratic cost, where these conditions can, indeed, be verified.
Note that we do not require any assumption on the distributions of the processes $X^i$, $i \in \{1,\ldots,N\}$, except for finiteness of second moments.

\begin{example}  \label{ex:2-bary} 
Consider the setting of {\rm \Cref{rem:baryeucnorm}} with the particular choice $p=2$. We have already seen that there exists a process $\bar{\Y} \in {\rm FP}_2$ solving
\[ \inf_{\Y \in {\rm FP}}\sum_{i=1}^N \lambda^i\Ac\Wc_2^2(\X^i,\Y). \] 
Moreover, by {\rm\Cref{example:separ.cost}} we know that {\rm\Cref{ass.baryandmulti}} is satisfied. We can readily see that the function 
\[ \phi^0(x^1,\ldots, x^N)\coloneqq\sum_{i=1}^N 2\lambda^i x^i,\quad (x^1,\ldots, x^N) \in (\R^{d\cdot T})^N\] 
is the only element in $\Phi^0$ for the choice $c^i=\lambda^i \lVert \cdot - \cdot \rVert_2^2.$ That is, we have
\[ \inf_{\Y \in {\rm FP}}\sum_{i=1}^N \lambda^i\Ac\Wc_2^2(\X^i,\Y)=\inf_{\pi \in \cplmc(\X^1,\ldots,\X^N)} \E_\pi \left[ \sum_{i=1}^N \lambda^i \lVert X^i-\phi^0(X^1,\ldots,X^N) \rVert^2_2\right].   \]
Moreover, every solution to the right-hand side yields a solution to the left-hand side and \emph{vice versa}  by {\rm \Cref{thm:con_of_form}}.
\end{example}
\subsection{Duality} \label{sec:dual_AW}
Finally, we discuss the dual problem to bicausal barycenters. Let us first mention the following observation, which readily follows from the results of the previous section.

\begin{proposition} Under {\rm\Cref{ass.baryandmulti}}, let $c$ be defined by \eqref{eqn:bary_multi_c} and assume that there exist $\ell^i \in \L^1({\rm Law}_{\P^i}(X^i))$, $i \in \{1,\ldots,N\}$, such that $c\leq \bigoplus_{i=1}^N \ell^i$. Then we have $B_{c^{1:N}}^{\rm bc}=D_{c}^{\rm mc}.$
\end{proposition}
\begin{proof} We have by \Cref{thm:con_of_form,thm:dual_multicausal} that
    $B_{c^{1:N}}^{\rm bc}=V_{c}^{\rm mc}=D_{c}^{\rm mc}.$
\end{proof}

 In order to prove the duality result for the bicausal barycenter problem, we briefly define the spaces of `canonical filtered processes' introduced in \citep{BaBePa21}, that is, the canonical space for the so-called `information process'. Let ${\rm FP}(\Yc) \ni \Y=\big(\Omega^\Y,\Fc^\Y,\F^\Y,\P^\Y,Y \big).$ We set inductively backward in time
\begin{align*}
{\rm ip}_T(\Y)&\coloneqq Y_T \in \Zc_T, &\text{where}\quad \Zc_T &\coloneqq \Yc_T, \\
{\rm ip}_{t}(\Y)&\coloneqq \big(Y_t, {\rm Law}_{\P^\Y}({\rm ip}_{t+1}(\Y)\,\big\vert\, \Fc^\Y_t ) \big)\in \Zc_t, &\text{where}\;\quad \Zc_t &\coloneqq \big(\Yc_t,\Pc(\Zc_{t+1})\big).
\end{align*}
We note that, by the construction, for every $t$ the space $\Zc_t$ is Polish and thus the conditional laws above exist. Moreover, we define $\Zc\coloneqq \Zc_{1:T},$ $\F^\Zc=(\Fc^\Zc_t)_{t=1}^T$ as the canonical filtration on $\Zc,$ $\Fc^\Zc\coloneqq \Fc^\Zc_T$, and $\Q\coloneqq {\rm Law}_{\P^\Y}({\rm ip}(\Y))\in \Pc(\Zc).$ Finally, we write ${\rm ip}_t(\Y)=(\hat{Y}_t,P_t) \in \Zc_t=\big(\Yc_t,\Pc(\Zc_{t+1})\big)$ 
for the two coordinates of ${\rm ip}(\Y),$ and we set

\[ \hat \Y\coloneqq \Big(\Zc,\Fc^\Zc,\F^\Zc,\Q, \hat Y \Big).\]
We have that $\hat \Y \sim \Y$ and call $\hat \Y$ the canonical representative of $\Y.$ Conversely, any measure $\nu \in \Pc(\Zc_1)$ defines a filtered process by 
\[ \Y^\nu\coloneqq \big(\Zc,\Fc^\Zc,\F^\Zc,\nu \otimes K_2,Y \big),\]
where for $t \in \{2,\dots,T\}$ we define the kernel $K_t : \Zc_{1:t-1} \longrightarrow \Pc(\Zc_{t:T})$ inductively backward in time by 
\[
    K_T(z_{1:T-1};\mathrm d z_T) \coloneqq p_{T-1}(\mathrm d z_T), \; K_t(z_{1:t-1};\mathrm d z_{t:T}) \coloneqq p_{t-1}(\mathrm dz_t) K_{t+1}(z_{1:t}; \mathrm dz_{t+1:T}),\; z_{1:T}=(y_{1:T},p_{1:T}) \in \Zc.
\] 
We define the set of test functions:
 \begin{align*}   
     \mathfrak{G}^{i,{\rm bc}}&\coloneqq \bigg\{ G^i(\omega^i,z)=\sum_{t=2}^T a^i_t(\omega^i_{1:t},z_{1:t-1})-\int a^i_t(\omega^i_{1:t-1},\tilde{\omega}^i_t,z_{1:t-1}) \P^i_{t,\omega^i_{\smallfont{1}\smallfont{:}\smallfont{t}\smallfont{-}\smallfont{1}}}(\mathrm{d}\tilde{\omega}^i_{t})\,\bigg\vert\,\\
     &\hspace{1.5cm}a^i_t \text{ is Borel measurable and } a^i_t(\omega^i_{1:t-1},\,\cdot\,,z_{1:t-1}) \in  \L^1\big(\Bc(\Omega^i_t),\P^i_{t,\omega^i_{\smallfont{1}\smallfont{:}\smallfont{t}\smallfont{-}\smallfont{1}}}\big),\, t \in\{2,\ldots,T\} \bigg\}, \\
    \mathfrak{H}^{i,{\rm bc}} &\coloneqq
    \bigg\{ 
        H^i(\omega^i,z)=\sum_{t=2}^T
        a_t^i(\omega^i_{1:t-1},z_{1:t}) - \int a_t^i(\omega^i_{1:t-1},z_{1:t-1},\tilde z_t) K_t(z_{1:t-1};\mathrm d\tilde z_t) \,\bigg\vert\, \\ 
        &\hspace{1.5cm}a_t^i \text{ is Borel measurable and } a_t^i(\omega^i_{1:t-1},z_{1:t-1},\,\cdot\,) \in \L^1\big(\Bc(\Zc_t),K_t(z_{1:t-1};\,\cdot\,)\big),\, t \in\{2,\ldots,T\} \bigg\}.
\end{align*}
We similarly define $\mathfrak{G}^{i,{\rm bc}}_{cts} \subseteq \mathfrak{G}^{i,{\rm bc}}$ and $\mathfrak{H}^{i,{\rm bc}}_{cts} \subseteq \mathfrak{H}^{i,{\rm bc}},$ consisting of the functions that additionally satisfy $a_t^i \in \Cc_b(\Omega^i_{1:t}\times \Yc_{1:t-1})$ and $a_t^i \in \Cc_b(\Omega^i_{1:t-1}\times \Yc_{1:t}),$ respectively. Finally, we set
\begin{align*}  \Phi^{{\rm bc}}&\coloneqq \bigg\{ (f^i)_{i=1}^N \in \prod_{i=1}^N \L^1(\Fc_1^i,\P^i)\ \bigg| \  \forall i \in \{1,\ldots,N\},\, \exists g^i: \Zc_1 \longrightarrow \R \text{ Borel measurable, } \\
&\hspace{0.5cm}\exists G^i \in \mathfrak{G}^{i,{\rm bc}},\, H^i \in \mathfrak{G}^{i,{\rm bc}}: f^i(\omega^i_1) + g^i(z_1) \leq c^i(X^i(\omega^i),y) + G^i(\omega^i,y)+H(\omega^i,z),\;\sum_{i=1}^N g^i(z_1)=0 \bigg\}, \\
\Phi^{{\rm bc}}_{cts}&\coloneqq \bigg\{ (f^i)_{i=1}^N \in \prod_{i=1}^N \Cc_b(\Omega_1^i)\ \bigg| \ \forall i \in \{1,\ldots,N\},\, \exists g^i \in \Cc_b(\Zc_1),\,\exists G^i \in \mathfrak{G}^{i,{\rm bc}}_{cts},\, H^i \in \mathfrak{G}^{i,{\rm bc}}_{cts}: \\
&\hspace{2.5cm} f^i(\omega^i_1) + g^i(z_1) \leq c^i(X^i(\omega^i),y) + G^i(\omega^i,y)+H(\omega^i,z),\;\sum_{i=1}^N g^i(z_1)=0 \bigg\}.
\end{align*}

\begin{theorem}[Duality for bicausal barycenters] \label{thm:duality_AW_bary}
    Assume that $\Yc$ is compact and that $c^i$, $i \in \{1,\ldots,N\},$ are lower-semicontinuous and lower-bounded.
    Then we have
    \begin{align*}
       B_{c^{1:N}}^{\rm bc} 
        = \sup_{f^{1:N} \in \Phi^{{\rm bc}}_{cts}} \sum_{i = 1}^N \int f^i(\omega^i_1) \P^i(\mathrm{d}\omega^i_1) \eqqcolon D_{c^{1:N},cts}^{\rm bc}.
    \end{align*}
\end{theorem}
\begin{proof} By \Cref{ass:continuity}, we have that the successive disintegrations of $\P^i$ are continuous. Note that refining the topology on $\Omega^i,$ $ i \in\{1,\ldots,N\}$ clearly does not change lower-semicontinuity of the functions $c^i,$ see \Cref{lem:topology}.
As $\Yc$ is compact, it is readily seen from the backward construction that $\Zc$ is also compact. We have
\begin{align*}
        &\inf_{\Y \in {\rm FP}(\Yc)} \sum_{i = 1}^N \Ac\Wc_{c^i}(\X^i,\Y) =
        \inf_{\nu \in \Pc(\Zc_1)} \sum_{i = 1}^N \Ac\Wc_{c^i}(\X^i,\Y^\nu) \\
        &=\inf_{\nu \in \Pc(\Zc_1)} \sum_{i = 1}^N \sup \bigg\{ \int f^i_1 \mathrm{d}\P^i+\int g_1^i \mathrm{d}\nu \,\bigg\vert\,\\
        &\hspace{2.5cm} f^i \in \Cc_b(\Omega^i),g^i \in \Cc_b(\Zc_1), G^i \in \mathfrak{G}^{i,{\rm bc}}_{cts},\, H^i \in \mathfrak{G}^{i,{\rm bc}}_{cts},\,i \in \{1,\dots,N\} : f^i+g^i \leq c^i+G^i+H^i \bigg\} \\
        &=\inf_{\nu \in \Pc(\Zc_1)} \sup \bigg\{ \sum_{i = 1}^N \int f^i_1 \mathrm{d}\P^i +\int \sum_{i=1}^N g_1^i \mathrm{d}\nu \,\bigg\vert\,\\
        &\hspace{2.5cm} f^i \in \Cc_b(\Omega^i),g^i \in \Cc_b(\Zc_1), G^i \in \mathfrak{G}^{i,{\rm bc}}_{cts},\, H^i \in \mathfrak{G}^{i,{\rm bc}}_{cts} ,\,i \in \{1,\dots,N\} : f^i+g^i \leq c^i+G^i+H^i \bigg\} \\
        &= \sup \bigg\{ \sum_{i = 1}^N \int f^i_1 \mathrm{d}\P^i + \inf_{\nu \in \Pc(\Zc_1)} \int \sum_{i=1}^N g_1^i \mathrm{d}\nu  \,\bigg\vert\,\\
        &\hspace{2.5cm} f^i \in \Cc_b(\Omega^i),g^i \in \Cc_b(\Zc_1), G^i \in \mathfrak{G}^{i,{\rm bc}}_{cts},\, H^i \in \mathfrak{G}^{i,{\rm bc}}_{cts} ,\,i \in \{1,\dots,N\} : f^i+g^i \leq c^i+G^i+H^i \bigg\} \\
        &= \sup \bigg\{ \sum_{i = 1}^N \int f^i_1 \mathrm{d}\P^i + \min_{z \in \Zc} \sum_{i=1}^N g_1^i(z)  \,\bigg\vert\,\\
        &\hspace{2.5cm} f^i \in \Cc_b(\Omega^i),g^i \in \Cc_b(\Zc_1), G^i \in \mathfrak{G}^{i,{\rm bc}}_{cts},\, H^i \in \mathfrak{G}^{i,{\rm bc}}_{cts},\,i \in \{1,\dots,N\}  : f^i+g^i \leq c^i+G^i+H^i \bigg\} \\
        &= \sup \bigg\{ \sum_{i = 1}^N \int f^i_1 \mathrm{d}\P^i \,\bigg\vert\, f^i \in \Cc_b(\Omega^i),g^i \in \Cc_b(\Zc_1), G^i \in \mathfrak{G}^{i,{\rm bc}}_{cts},\, H^i \in \mathfrak{G}^{i,{\rm bc}}_{cts},\,i \in \{1,\dots,N\}  :\\
        &\hspace{10.7cm} f^i+g^i \leq c^i+G^i+H^i \text{ and } \sum_{i=1}^N g^i \geq 0 \bigg\} \\
        &=\sup \bigg\{ \sum_{i = 1}^N \int f^i_1 \mathrm{d}\P^i \,\bigg\vert\, f^i \in \Cc_b(\Omega^i),g^i \in \Cc_b(\Zc_1), G^i \in \mathfrak{G}^{i,{\rm bc}}_{cts},\, H^i \in \mathfrak{G}^{i,{\rm bc}}_{cts},\,i \in \{1,\dots,N\}  :\\
        &\hspace{10.7cm} f^i+g^i \leq c^i+G^i+H^i \text{ and } \sum_{i=1}^N g^i = 0 \bigg\} \\
        &=\sup_{f^{1:N} \in \Phi^{{\rm bc}}_{cts}} \sum_{i = 1}^N \int f^i(\omega^i_1) \P^i(\mathrm{d}\omega^i_1),
\end{align*}
where the second equality follows from duality for bicausal optimal transport with lower-semicontinuous cost function. The fourth equality follows from the minimax theorem. Indeed, it is easy to verify that the map
\begin{multline*} \Cc_b(\Zc_1) \times \Pc(\Zc_1) \ni ( \varphi,\nu ) \longmapsto  \int \varphi \mathrm{d}\nu + \sup \bigg\{ \sum_{i = 1}^N \int f^i \mathrm{d}\P^i \,\bigg\vert\, f^i \in \Cc_b(\Omega^i),g^i \in \Cc_b(\Zc_1),\\
G^i \in \mathfrak{G}^{i,{\rm bc}}_{cts},\, H^i \in \mathfrak{G}^{i,{\rm bc}}_{cts} : f^i+g^i \leq c^i+G^i+H^i \text{ and } \sum_{i=1}^N g^i = \varphi \bigg\}
\end{multline*}is convex and continuous in $\nu$ and concave in $\varphi$. This together with the compactness of $\Pc(\Zc_1)$ guarantees that we can employ the minimax theorem \cite[Theorem 2]{minimax}. Furthermore, in the fifth equality, we have replaced $(f^1,g^1)$ with $(f^1+\min_{z \in \Zc} \sum_{i=1}^N g^i(z_1),g^1-\min_{z \in \Zc} \sum_{i=1}^N g^i(z))$, and the remaining equalities are straightforward.
\end{proof}

We conclude this section with a result on duality and dual attainment for general measurable functions, which was proved in \cite{KrPa23}. We remark that \cite{KrPa23} provides a more general result involving non-compact space $\Yc.$ We present only the compact case here for the sake of simplicity.

\begin{theorem}[General duality for bicausal barycenters] 
    Assume that $\Yc$ is compact and, for  $i \in \{1,\ldots,N\}$, $c^i$ are measurable and $\ell^i \in \L^1({\rm Law}_{\P^i}(X^i))$ are such that $|c^i(x,y)| \leq  \ell^i(x),$ $(x,y) \in \Xc \times \Yc$.
    Then, we have
    \begin{align*}
        B_{c^{1:N}}^{\rm bc} 
        = \sup_{f^{1:N} \in \Phi^{{\rm bc}} } \sum_{i = 1}^N \int f^i(\omega^i_1) \P^i(\mathrm{d}\omega^i_1).
    \end{align*} Moreover, the right-hand side is attained.
\end{theorem}

\section{Causal barycenters} \label{sec:CWbary}
In this section, we study the barycenter problem with causal constraints. That is to say, we consider the minimization problem
\begin{equation*} \inf_{\Y\in {\rm FP}(\Yc)} B_{c^{1:N}}^{\rm c}(\Y),\end{equation*}
where $c^i : \Xc \times \Yc \longrightarrow \R$, $i \in \{1,\ldots, N\}$, are given measurable functions and 
\[
B_{c^{1:N}}^{\rm c}(\Y)\coloneqq \sum_{i=1}^N \Cc\Wc_{c^i}(\X^i,\Y), \quad \Y \in {\rm FP}(\Yc).
\]

The problem introduced above has, generally speaking,  substantially different properties than the bicausal barycenter problem. First, even if the cost function is symmetric, the causal optimal transport is inherently asymmetric due to the causality constraint. As a consequence, we lack the multimarginal optimal transport formulation, and other properties --such as the dynamic programming principle-- also differ.

\medskip
As in the case of bicausal barycenters, here as well we can think of the filtered processes $\X^1,\dots,\X^N$ as models for beliefs of $N$ different experts.
Now, however, agents aim at finding a consensus process but not a consensus filtration. This can be thought of as agents sharing their beliefs and eventually having to agree on a stochastic process but not the underlying flow of information.
This will result in a model endowed with canonical filtration (canonical filtered process), as the minimal requirement of the causal barycenter process is being adapted.

\begin{remark}[Anticausal barycenters]
Let us note that the barycenter problem which is causal in the other direction is not of much interest since it coincides with the barycenter problem in standard optimal transport. Indeed, consider the anticausal barycenter problem, that is, 
\begin{equation}\label{eq:acb}
\inf_{\Y\in {\rm FP}(\Yc)} \sum_{i=1}^N \Cc\Wc_{c^i}(\Y,\X^i)\end{equation}
and let {\rm\Cref{assAdaptBary}} hold. 
From {\rm\Cref{AdaptBaryExist}} (applied to $\tilde T=1$ and $\tilde \Xc_1\coloneqq \Xc_{1:T}$), we conclude that there exists a solution, say $\mu^\star \in \Pc(\Yc)$, to the standard optimal transport barycenter problem \begin{equation} \label{eqn:wass_bary}
\inf_{\mu \in \Pc(\Yc)} \sum_{i=1}^N \Wc_{c^i}(\mu,\P^i),\; {\rm where}\; \Wc_{c^i}(\mu,\P^i)\coloneqq \inf_{\pi^i \in \cpl(\mu,\P^i)} \int c^i(y,X^i(\omega^i)) \pi^i(\mathrm{d}y, \mathrm{d}\omega^i), \end{equation} 
as well as optimal transport plans, say $\pi^{\star,i} \in \cpl(\mu^\star,\P^i)$, $i \in \{1,\ldots,N \}$. Consider the disintegration $\pi^{\star,i}(\mathrm{d}y, \mathrm{d}\omega^i)=K^i(y,\mathrm{d}\omega^i)\mu^\star(\mathrm{d}y)$ and define the filtered process
\begin{equation*} 
\bar{\X}\coloneqq \left( \overline{\Omega} \times \Yc, \overline{\Fc}_T \otimes \Fc^Y_T, \overline{\F} \otimes \F^Y,\gamma^\star,\bar{Y}\right), \end{equation*} 
where $(\Yc,\Fc^Y_T,\F^Y)$ is as in {\rm \Cref{rem:plain}} and
\[ \bar{Y}(\omega^1,\ldots,\omega^N,y)\coloneqq y,\;\gamma^\star(\mathrm{d}\overline{\omega},\mathrm{d}y)\coloneqq \bigg(\prod_{i=1}^N K^i(y;\mathrm{d}\omega^i)\bigg)\mu^\star(\mathrm{d}y).\]

 It is obvious that the projection ${\rm proj}^i : \overline{\Omega} \times \Yc  \longrightarrow \Omega^i$ is $\overline{\F} \otimes \F^Y$-adapted and hence $((\overline \omega,y) \mapsto (y, \omega^i))_{\#} \gamma^\star \in \cplc(\bar{\Y},\X^i).$ It follows that
\begin{equation*}  
\sum_{i=1}^N \Cc\Wc_{c^i}(\bar{\Y},\X^i) \leq \sum_{i=1}^N \int c^i(X^i \circ {\rm proj}^i,\bar{Y} )  \mathrm{d} \gamma^\star =\inf_{\mu \in \Pc(\Yc)} \sum_{i=1}^N  \Wc_{c^i}(\mu,\P^i)\leq \inf_{\Y\in {\rm FP}(\Yc)} \sum_{i=1}^N \Cc\Wc_{c^i}(\Y,\X^i).
\end{equation*}
Consequently, we obtain that $\bar{\Y}$ is an anticausal barycenter. Thus solving the anticausal barycenter problem \eqref{eq:acb} boils down to solving the standard optimal transport barycenter problem \eqref{eqn:wass_bary}.
\end{remark} 

\begin{example}[A causal analogy to the multicausal results in \Cref{thm:con_of_form} does \emph{not} hold] Let $T=2$ and let us consider two $\R$-valued processes $X^1=(Y,Y^{3})$ and $X^2=(0,Y^{3})$, together with their canonical filtrations. Let $Y \sim \Nc(0,1)$ and set $c^1(x,y)=c^2(x,y)=\frac{1}{2}\lVert x-y \rVert^2_2.$ We know from {\rm\Cref{ex:2-bary}} that the function
\[ \phi^0(x^1,x^2)\coloneqq x^1+x^2,\quad (x^1,\ldots, x^N) \in (\R^{T})^N\] is a unique element in $\Phi^0.$ It can be readily seen that the only element $\pi \in \cpl(\X^1,\X^2)$ such that  $({\rm proj}^1,{\rm id})_{\#} \pi \in \cplc(\X^1,\Y^{\pi,\phi^0})$ and $({\rm proj}^2,{\rm id})_{\#} \pi \in \cplc(\X^2,\Y^{\pi,\phi^0})$ is the product coupling. That is, 
\[ X^1=(Y,Y^3),\,\,X^2=(0,{Y'}^3),\,\,Y^{\pi,\phi^0}=\left(Y,Y^3+{Y'}^3 \right), \] where $Y,Y' \stackrel{{\rm i.i.d.}}{\sim} \Nc(0,1).$ The transport cost is then \[ \sum_{i=1}^2 \frac{1}{2} \E_{\pi} \lVert X^i-Y^{\pi,\phi^0} \rVert^2=\frac{1}{2}\left( \E Y^2+ 2\E Y^{6} \right)=15.5.\]
If we consider the process $\tilde{Y}=(0,Y^3)$ with its canonical filtration, the causal transport cost is
\[ \sum_{i=1}^2 \frac{1}{2} \Cc\Wc_2^2(X^i,\tilde{Y})=\E Y^2=1.\]
It is thus clear that there is no barycenter of the form $\Y^{\pi,\phi^0}$ for some $\pi \in \cpl(\X^1,\X^2).$
\end{example}

\subsection{Existence of primal optimizers}
Let us start with the following observation. Due to the asymmetry of the causality constraint, it is easy to observe that, unlike in the bicausal case, the solution can always be chosen with the `smallest' possible filtration. This is formalized in the following remark.

\begin{remark} \label{rem:plain}
Let $\Y=\big( \Omega^\Y, \Fc^\Y, \F^\Y, \P^\Y, (Y_t)_{t=1}^T  \big)$ be a filtered process and  define \[\Y^\text{plain}\coloneqq \Big( \Omega^\Y, \Fc^Y, \F^Y, \P^\Y, (Y_t)_{t=1}^T  \Big),\] where $\Fc^Y_t\coloneqq \sigma\{ Y_s \,\vert\, s \in \{1,\ldots,t\} \}.$ It can be easily seen that $\cplc(\X^i,\Y)\subseteq \cplc(\X^i,\Y^\text{plain})$, $i \in \{1,\ldots, N\}$, and therefore we immediately obtain $B_{c^{1:N}}^{\rm c}(\Y)\geq B_{c^{1:N}}^{\rm c}(\Y^\text{plain}).$ In particular, if there is a solution to $B_{c^{1:N}}^{\rm c}$, it can be chosen with the canonical filtration. This is not the case for bicausal barycenters, see {\rm\Cref{rem:filtr_bary}.}

\medskip In particular, we have the following. For $\mu \in \Pc(\Yc)$, we set
\[ \Y^\mu\coloneqq \big( \Yc, \Fc^Y_T, \F^Y,\mu,Y \big),\] where
$Y$ is the canonical process on $\Yc$, \emph{i.e.} $Y_t(y)=y_t$ for $(t,y) \in \{1,\ldots,T\} \times \Yc$, and $\F^Y=(\Fc^Y_t)_{t=1}^T$ is the canonical filtration. We then have
\begin{equation} \label{eqn:plain} \inf_{\Y\in {\rm FP}(\Yc)} B_{c^{1:N}}^{\rm c}(\Y)=\inf_{\mu \in \Pc(\Yc)} B_{c^{1:N}}^{\rm c}(\Y^\mu).\end{equation}
\end{remark}

The attainment can be derived analogously as in the bicausal case. 

\begin{theorem} \label{existCausBary} Let {\rm\Cref{assAdaptBary}} hold with $B_{c^{1:N}}^{\rm bc}(\Y^0)<\infty$ replaced by $B_{c^{1:N}}^{\rm c}(\Y^0)<\infty$. Then there exists a solution to $B_{c^{1:N}}^{\rm c}.$ Moreover, any solution lies in ${\rm FP}_1(\Yc).$
\end{theorem}
\begin{proof} The statement can be proved by analogous arguments as in the proof of \Cref{AdaptBaryExist}. Due to \eqref{eqn:plain}, we just need to verify lower semicontinuity of the map $\Pc(\Yc) \ni \mu \longmapsto \Cc\Wc_{c^{i}}(\X^i,\Y^\mu).$ This can be done analogously as in \eqref{eqn:lsc}. We have to verify that if $\Kc \subseteq \Pc(\Yc)$ is compact, then the set $\bigcup_{\mu \in \Kc} \cplc(\X^i,\Y^\nu)$ is compact for every fixed $i \in \{1,\ldots,N\}.$ As in the proof of \Cref{lem:comp_multic} compactness of $\bigcup_{\mu \in \Kc} \cpl(\X^i,\Y^\mu)$ is immediate, and so we have to show that any cluster point is a causal coupling. This follows from the fact that a coupling $\pi \in \bigcup_{\mu \in \Pc(\Yc)}\cplc(\X^i,\Y^\mu)$ is causal if and only if 
\[ \int \bigg[ \sum_{t=2}^T a^i_t(\omega^i_{1:t},y_{1:t-1})-\int a^i_t(\omega^i_{1:t-1},\tilde{\omega}^i_t,y_{1:t-1}) \P^i_{t,\omega^i_{\smallfont{1}\smallfont{:}\smallfont{t}\smallfont{-}\smallfont{1}}}(\mathrm{d}\tilde{\omega}^i_{t}) \bigg] \pi(\mathrm{d}\omega^i,\mathrm{d}y)=0 \]
for every $a_i^t : \Omega^i \times \Yc \longrightarrow \R$ continuous and bounded, $t \in \{2,\ldots,T\}.$ As the function 
\[ (\omega^i,y) \longmapsto \sum_{t=2}^T a^i_t(\omega^i_{1:t},y_{1:t-1})-\int a^i_t(\omega^i_{1:t-1},\tilde{\omega}^i_t,y_{1:t-1}) \P^i_{t,\omega^i_{\smallfont{1}\smallfont{:}\smallfont{t}\smallfont{-}\smallfont{1}}}(\mathrm{d}\tilde{\omega}^i_{t})\] is continuous and bounded, the claim follows.
\end{proof}

\subsection{Duality} \label{sec:CWduality}
For every $i\in\{1,\ldots,N\}$, define the set of test functions
 \begin{align*}   
     \mathfrak{G}^{i,{\rm c}}&\coloneqq \bigg\{ G^i(\omega^i,y)=\sum_{t=2}^T a^i_t(\omega^i_{1:t},y_{1:t-1})-\int a^i_t(\omega^i_{1:t-1},\tilde{\omega}^i_t,y_{1:t-1}) \P^i_{t,\omega^i_{\smallfont{1}\smallfont{:}\smallfont{t}\smallfont{-}\smallfont{1}}}(\mathrm{d}\tilde{\omega}^i_{t})\,\bigg\vert\,\\
     &\hspace{1.5cm} a^i_t \text{ is Borel measurable and } a^i_t(\omega^i_{1:t-1},\,\cdot\,,y_{1:t-1}) \in  \L^1\big(\Bc(\Omega^i_t),\P^i_{t,\omega^i_{\smallfont{1}\smallfont{:}\smallfont{t}\smallfont{-}\smallfont{1}}}\big),\, t \in\{2,\ldots,T\} \bigg\}.
\end{align*}
We similarly define $\mathfrak{G}^{i,{\rm c}}_{cts} \subseteq \mathfrak{G}^{i,{\rm c}},$ consisting of the functions that additionally satisfy $a_t^i \in \Cc_b(\Omega^i_{1:t}\times \Yc_{1:t-1})$.

We further set
\begin{align*}  \Phi^{{\rm c}}&\coloneqq \bigg\{ (f^i)_{i=1}^N \in \prod_{i=1}^N \L^1(\Fc_1^i,\P^i)\ \bigg| \  \forall i \in \{1,\ldots,N\},\, \exists g^i: \Yc \longrightarrow \R \text{ Borel measurable, } \\
&\hspace{2.5cm}\exists G^i \in \mathfrak{G}^{i,{\rm c}}: f^i(\omega^i_1) + g^i(y) \leq c^i(X^i(\omega^i),y) + G^i(\omega^i,y),\;\sum_{i=1}^N g^i(y)=0 \bigg\}
\end{align*} 
and 
\[D_{c^{1:N}}^{\rm{c}}\coloneqq \sup_{(f^i)_{i=1}^N \in \Phi^{\rm c}_{cts}} \sum_{i=1}^N \int f^i(\omega^i) \P^i(\mathrm{d}\omega^i).  \]

\begin{remark} \label{rem:weak_dual_bary} For any $\mu \in \Pc(\Xc)$, $\pi \in \cplc(\X^i,\X^\mu)$ and $(f^i)_{i=1}^N \in \Phi^{{\rm c}},$ we have
\begin{align*}
   \sum_{i=1}^N \int f^i(\omega^i)\P^i(\mathrm{d}\omega^i) &= \sum_{i=1}^N \bigg[ \int f^i(\omega^i)\P^i(\mathrm{d}\omega^i) + \int g^i(y) \mu(\mathrm{d}y) \bigg] \\ 
   &\leq \sum_{i=1}^N \int \Big[c^i(X^i(\omega^i),y) + G^i(\omega^i,y)\Big]\pi^i(\mathrm{d}\omega^i,\mathrm{d}y)=\sum_{i=1}^N \int c^i(X^i(\omega^i),y) \pi^i(\mathrm{d}\omega^i,\mathrm{d}y),
\end{align*} where $g^i$ and $G^i$ come from the definition of $\Phi^{{\rm c}}.$ This yields the weak duality $D_{c^{1:N}}^{{\rm c}} \leq B_{c^{1:N}}^{\rm{c}}.$
\end{remark}

\begin{remark}
If the space $\Yc$ is assumed to be compact, duality can be shown as in {\rm\Cref{thm:duality_AW_bary}}, with the obvious changes having been made. For illustration, we show a different proof, which follows similar steps as the proof of {\rm\cite[Proposition 2.2]{AgCa11}},  and does not require compactness of $\Yc$, but only the more general {\rm\Cref{assAdaptBary}}.
\end{remark}

\begin{theorem} \label{thm:Caus_duality} Let {\rm\Cref{assAdaptBary}} hold with $B_{c^{1:N}}^{\rm bc}(\Y^0)<\infty$ replaced by $B_{c^{1:N}}^{\rm c}(\Y^0)<\infty$. Then we have
\begin{equation} \label{eqn:dual_caus_bary}
B_{c^{1:N}}^{\rm{c}}= D_{c^{1:N}}^{\rm{c}}. 
\end{equation}
\end{theorem}
\begin{proof} We recall that, by \Cref{ass:continuity}, the maps $\omega^i_{1:t} \longmapsto \P^i_{t+1,\omega^\smallfont{i}_{\smallfont{1}\smallfont{:}\smallfont{t}}}$ are assumed to be continuous. It follows from \Cref{assAdaptBary}.$\ref{assAdaptItem3}$ that $\Yc$ is a locally compact space. We define
\[\Cc_{1}(\Yc)\coloneqq\bigg\{ g \in \Cc(\Yc) \,\bigg\vert\, \frac{g(\,\cdot\,)}{1+d(y^0,\,\cdot\,)} \in \Cc_0(\Yc) \bigg\},\qquad \lVert g \rVert_{\Cc_1}\coloneqq \sup_{y \in \Yc} \frac{\lvert g(y) \rvert }{1+d(y^0,y)},\] 
where $\Cc(\Yc)$ denotes the space of all continuous functions on $\Yc$, and $\Cc_0(\Yc)$ the space of all continuous functions on $\Yc$ that vanish at infinity, and where $y^0$ is from \Cref{assAdaptBary}. Using the Riesz--Markov--Kakutani representation, we identify the dual of $\Cc_{1}(\Yc)$ with
\[\Mc_1(\Yc)\coloneqq\{ \mu \in \Mc_f(\Yc) : (1+d(y^0,y))\mu(\mathrm{d}y) \in \Mc_f(\Yc) \},\]
where $\Mc_f(\Yc)$ denotes the set of all regular finite Borel signed measures on $\Yc.$ For any $g^i \in \Cc_{1}(\Yc)$, we define the functions $(g^i)^{c^i,G^i} : \Omega^i \longrightarrow \R \cup \{ \pm \infty\}$ by \begin{align*}
(g^i)^{c^i,G^i}(\omega^i)&\coloneqq\inf_{y \in \Yc} \big\{c^i(X^i(\omega^i),y)-g^i(y) +G^i(\omega^i,y) \big\},\quad G^i \in \mathfrak{G}^{i,{\rm c}}_{cts}.  \end{align*}

We consider the problem
\[\tilde D=\sup \bigg\{ \sum_{i=1}^N \int (g^i)^{c^i,G^i} \mathrm{d}\P^i \,\bigg\vert\, G^i \in \mathfrak{G}^{i,{\rm c}}_{cts},\; g^i \in \Cc_{1}(\Yc),\; i \in \{1,\ldots,N\},\; \sum_{i=1}^N g^i=0 \bigg\}. \]
It is easy to see that 
\[\tilde D \leq D_{c^{1:N}}^{\rm{c}} \leq B_{c^{1:N}}^{\rm{c}}.\]
This is a consequence of weak duality \eqref{rem:weak_dual_bary} and the following argument:
Clearly, 
\[(g^i)^{c^i,G^i}(\omega^i) \leq c^i(X^i(\omega^i),y)-g^i(y)+G^i(\omega^i,y)\] with $G^i \in \mathfrak{G}^{i,{\rm c}}_{cts} \subseteq \mathfrak{G}^{i,{\rm c}}$. Moreover, since
\[ (g^i)^{c^i,G^i}(\omega^i) \leq c^i(X^i(\omega^i),y^0)-g^i(y^0)+G^i(\omega^i,y^0), \]
and the right-hand side is integrable with respect to $\P^i$ by \Cref{assAdaptBary}.$\ref{assAdaptItem3}$, we have that $\int (g^i)^{c^i,G^i} \mathrm{d}\P^i$ is well-defined, and either $(g^i)^{c^i,G^i} \in \L^1(\P^i)$ or $\int (g^i)^{c^i,G^i} \mathrm{d}\P^i=-\infty$. Actually, as in \cite[Proposition 4.21]{KrPa23} and using \Cref{assAdaptBary}.$\ref{assAdaptItem3}$, one can even show that $(g^i)^{c^i,G^i} \in \L^1(\P^i)$.

\medskip To further shorten the notation, we set 
\begin{align*}
 S^i(g^i)&\coloneqq-\sup_{G^i \in \mathfrak{G}^{i,{\rm c}}_{cts}} \bigg\{ \int (g^i)^{c^i,G^i}(\omega^i) \P^i(\mathrm{d}\omega^i) \bigg\},\quad g^i \in \Cc_{1}(\Yc),\; i \in \{1,\ldots,N\},   \\
 S(g)&\coloneqq\inf \bigg\{ \sum_{i=1}^N S^i(g^i) \,\bigg\vert\, g^i \in \Cc_{1}(\Yc),\; \sum_{i=1}^N g^i=g \bigg\},\quad g \in \Cc_{1}(\Yc).
\end{align*}
It is then easy to see that $\tilde D=-S(0).$ We next show that $S^i : \Cc_{1}(\Yc) \longrightarrow [-\infty, \infty)$, $i \in \{1,\ldots,N\}$, and $S : \Cc_{1}(\Yc) \longrightarrow [-\infty,\infty]$ are convex. We start by fixing $i \in \{1,\ldots,N\}$ and showing convexity of $S^i$. Let $\alpha \in (0,1)$, $g^i, h^i \in \Cc_{1}(\Yc)$ and $G^i,H^i \in \mathfrak{G}^{i,{\rm c}}_{cts}.$ We have
\[ c^i(X^i,\,\cdot\,)-[\alpha g^i + (1-\alpha)h^i] +[\alpha G^i + (1-\alpha)H^i]=\alpha[c^i(X^i,\,\cdot\,)- g^i+G^i] + (1-\alpha)[c^i(X^i,\,\cdot\,)- h^i+H^i]. \]
Taking infima yields
\[ (\alpha g^i + (1-\alpha)h^i)^{c^i,\alpha G^i + (1-\alpha)H^i} \geq \alpha (g^i)^{c^i,G^i} + (1-\alpha) (h^i)^{c^i,H^i}. \]
Since $\alpha G^i + (1-\alpha) H^i \in \mathfrak{G}^{i,{\rm c}}_{cts}$ and $G^i$ and $H^i$ were arbitrary, this in turn gives 
\[\sup_{G^i \in \mathfrak{G}^{i,{\rm c}}_{cts}} \bigg\{ \int (\alpha g^i+ (1-\alpha h^i))^{c^i,G^i} \mathrm{d}\P^i \bigg\} \geq \alpha \sup_{G^i \in \mathfrak{G}^{i,{\rm c}}_{cts}} \bigg\{ \int (g^i)^{c^i,G^i} \mathrm{d}\P^i \bigg\} + (1-\alpha)\sup_{G^i \in \mathfrak{G}^{i,{\rm c}}_{cts}} \bigg\{ \int (h^i)^{c^i,G^i} \mathrm{d}\P^i \bigg\},  \] 
showing convexity of $S^i.$

\medskip To prove convexity of $S$, let  $\alpha \in (0,1)$, $g,h \in \Cc_{1}(\Xc)$ and $g^i,h^i \in \Cc_{1}(\Yc), i \in \{1,\ldots,N\}$, be such that $\sum_{i=1}^N g^i=g$ and $\sum_{i=1}^N h^i=h.$ Clearly,
\[ \sum_{i=1}^N S^i(\alpha g^i+(1-\alpha)h^i) \leq \alpha \sum_{i=1}^N S^i(g^i)+ (1-\alpha)\sum_{i=1}^N S^i(h^i) \] by convexity of $S^i.$ Because $\sum_{i=1}^N (\alpha g^i + (1-\alpha)h^i) = (\alpha g + (1-\alpha)h)$ and $g^i$ and $h^i$, $i \in \{1,\ldots,N\}$, were arbitrary, we obtain
\[ S(\alpha g + (1-\alpha)h)\leq \alpha S( g )+(1-\alpha)S(h).\]
This shows convexity of $S.$ The main part of the proof follows.

\medskip For simplicity, assume that the functions $c^i$, $i \in \{1,\ldots,N\}$, are non-negative, and denote by $S^{i,\star}$ the Legendre-Fenchel transform of $S^i.$ We have, for $\mu \in \Mc_1(\Yc)$,
\begin{align*}
S^{i,\star}(\mu)&= \sup_{g^i \in \Cc_{1}(\Yc)} \bigg\{ \int g^i(y) \mu^i(\mathrm{d}y) - S^i(g^i) \bigg\} =\sup_{\substack{g^i \in \Cc_{1}(\Yc) \\ G^i \in \mathfrak{G}^{i,{\rm c}}_{cts}}}  \bigg\{ \int g^i(y) \mu^i(\mathrm{d}y) + \int (g^i)^{c^i,G^i}(\omega^i) \P^i(\mathrm{d}\omega^i)  \bigg\}.
\end{align*}
For $\mu \in \Pc(\Xc) \cap \Mc_1(\Xc)$, we thus have
\[ S^{i,\star}(\mu)=\Cc\Wc_{c^i}(\X^i,\X^\mu), \] 
by the duality result for causal optimal transport, see \cite[Section 5]{EcPa22}. If $\mu$ is not non-negative or $\mu(\Xc) \neq 1$, one can easily verify, as in \cite[Lemma 2.1]{AgCa11}, that $S^{i,\star}(\mu)\geq \sup_{g^i \in \Cc_{1}(\Yc)}  \Big\{ \int g^i \mathrm{d}\mu + \int (g^i)^{c^i,0} \mathrm{d}\P^i\Big\}=\infty.$

\medskip
As already noticed, we have $\tilde D= -S(0).$ From what shown above and from \Cref{existCausBary}, we further deduce
\begin{multline*}
    B_{c^{1:N}}^{{\rm c}}=\inf_{\mu \in \Pc(\Yc) \cap \Mc_1(\Yc)} \sum_{i=1}^N \Cc\Wc_{c^i}(\X^i,\Y^\mu)=\inf_{\mu \in \Mc_1(\Yc)} \sum_{i=1}^N S^{i,\star}(\mu) \\
    =-\sup_{\mu \in \Mc_1(\Yc)} \Big\{0 -\sum_{i=1}^N S^{i,\star}(\mu) \Big\}=-\Big( \sum_{i=1}^N S^{i,\star} \Big)^\star (0).
\end{multline*} 
Straightforward calculations show that the Legendre-Fenchel transform $S^\star$ of $S$ satisfies $S^\star=\sum_{i=1}^N S^{i,\star}$, and thus it suffices to show that $S^{\star \star}(0)=S(0).$ We have that $S$ is convex. Therefore, according to \cite[Proposition 3.3]{Ekeland_convex}, it is sufficient to show that $S$ is lower-semicontinuous and admits a continuous affine minorant at $0.$ For any $g \in \Cc_{1}(\Yc)$ and $g^i \in \Cc_{1}(\Yc)$, $i \in \{1,\ldots,N\}$, such that $g=\sum_{i=1}^N g^i$, we have that
\begin{align*} 
    S^i(g^i)&= \inf_{G^i \in \mathfrak{G}^{i,{\rm c}}_{cts}} \bigg\{ \int \sup_{y \in \Yc} \{ -c^i(X^i(\omega^i),y)+g^i(y)-G^i(\omega^i,y) \} \P^i(\mathrm{d} \omega^i) \bigg\} \\
    &\geq \inf_{G^i \in \mathfrak{G}^{i,{\rm c}}_{cts}} \bigg\{ \int -c^i(X^i(\omega^i),y_0)+g^i(y^0)-G^i(\omega^i,y_0) \P^i(\mathrm{d} \omega^i) \bigg\}\\
    &=\inf_{G^i \in \mathfrak{G}^{i,{\rm c}}_{cts}} \bigg\{ g^i(y^0) - \int c^i(X^i(\omega^i),y^0) \P^i(\mathrm{d} \omega^i) \bigg\}\\
    &=g^i(y^0) - \int c^i(X^i(\omega^i),y^0) \P^i(\mathrm{d} \omega^i). 
\end{align*}
Thus,
\[ S(g)\geq g(x_0) - \sum_{i=1}^N \int c^i(X^i(\omega^i),y^0)  \P^i(\mathrm{d} \omega^i),  \] which gives a continuous affine minorant. Further, let $K>0$ be such that $K\leq 1/C$, where $C$ is given in \Cref{assAdaptBary}. For every $(x,y) \in \Xc \times \Yc$, using \Cref{assAdaptBary}.$\ref{assAdaptItem2}$, we obtain
\begin{align*} -c^1(x,y)+K(1+d_{\Yc}(y,y^0)) &\leq 1+\ell^1(x)- \frac{1}{C}d_{\Yc}(y,y^0)+K(1+ d(y^0,y)) \\
&\leq 1+K+\ell^1(x)+\Big(K-\frac{1}{C} \Big) d_{\Yc }(y,y^0) \leq 1+K+\ell^1(x).
\end{align*} 
Thus, if $g \in \Cc_{1}(\Yc)$ is a function satisfying $\lVert g \rVert_{\Cc_{1}} \leq K$, as $c^i$ is non-negative we have
\begin{align*} S(g) &\leq S^1(g)+ \sum_{i=2}^N S^i (0) \\
&= \inf_{G^1 \in \mathfrak{G}^{1,{\rm c}}_{cts}} \int \sup_{y \in \Yc} \{- c^1(X^1(\omega^1),y)+g(y) -G^1(\omega^1,y)\} \P^1(\mathrm{d} \omega^1) \\
&\hspace{1cm}+ \sum_{i=2}^N \inf_{G^i \in \mathfrak{G}^{i,{\rm c}}_{cts}} \int \sup_{y \in \Xc} \{- c^i(X^i(\omega^i),y) -G^i(\omega^i,y)\} \P^i(\mathrm{d} \omega^i) \\
&\leq \int \sup_{y \in \Yc} \{-c^1(X^1(\omega^1),y)+g(y)\} \P^1(\mathrm{d} \omega^1) + \sum_{i=2}^N \int \sup_{y \in \Yc} \{-c^i(X^i(\omega^i),y)\} \P^i(\mathrm{d} \omega^i) \\
&\leq \int \sup_{y \in \Yc} \{-c^1(X^1(\omega^1),y)+K(1+d(y^0,y))\} \P^1(\mathrm{d} \omega^1)\\
& \leq \int \big[1+K+\ell^1(X^1(\omega^1))\big] \P^1(\mathrm{d} \omega^1)<\infty.
\end{align*}
We conclude that $S$ is a convex function bounded from above on a neighborhood of $0$ which does not attain the value $-\infty.$ Thus, by \cite[Proposition 2.5]{Ekeland_convex} it is continuous at $0.$ This concludes the proof. 
\end{proof}

 As the causal barycenter problem is inherently asymmetric, and we work in a non-dominated setting, in order to achieve dual attainment, we need to relax the admissibility in the dual problem. More specifically, we consider a formulation in which, for every $\mu$, there exist dual potentials $G^i, g^i$, which may depend on $\mu$. This relaxation is unnecessary for the bicausal problem and the classical OT problem due to their more rigid structure. We refer to the work \cite{KrPa23} for more details. Moreover, we note that such relaxation is not required for dual attainment if one assumes the continuum hypothesis, which allows us to `aggregate' certain functions, see \cite[Theorem 4.28]{KrPa23}.

\medskip We set
\begin{align*}  \Phi^{{\rm c}}_{rel}&\coloneqq \bigg\{ (f^i)_{i=1}^N \in \prod_{i=1}^N \L^1(\Fc_1^i,\P^i)\ \bigg| \ \forall \mu \in \Pc(\Yc),\,  \forall i \in \{1,\ldots,N\},\, \exists g^i: \Yc \longrightarrow \R \text{ Borel measurable, } \\
&\hspace{1.5cm}\exists G^i \in \mathfrak{G}^{i,{\rm c}}: f^i(\omega^i_1) + g^i(y) \leq c^i(X^i(\omega^i),y) + G^i(\omega^i,y)\text{ and }\sum_{i=1}^N g^i(y)=0\; \cplc(\X^i,\Y^\mu)\text{--q.s.} \bigg\}.
\end{align*}
Let us emphasize that, compared to the set $\Phi^{{\rm c}},$ the functions $g^i$ and $G^i$ are allowed to depend on the chosen measure $\mu \in \Pc(\Yc),$ and the inequality holds only $\cplc(\X^i,\Y^\mu)$--quasi-surely.

\begin{proposition} \label{thm:existbary_dual} Let $\Yc$ be a $\sigma$-compact space. For every $ i \in \{1,\ldots,N\},$ let $c^i$ be measurable and lower-bounded, $k : \Yc \longrightarrow \R$ be bounded on compacts, and $\ell^i \in \L^1({\rm Law}_{\P^i}(X^i))$ be such that $c^i(x, y) \leq \ell^i(x)) + k(y), (\omega^i, y) \in \Omega^i \times \Yc.$ Then, we have
\[B_{c^{1:N}}^{\rm{c}}=\sup_{(f^i)_{i=1}^N \in \Phi^{\rm c}_{rel}} \sum_{i=1}^N \int f^i(\omega^i) \P^i(\mathrm{d}\omega^i),\]
and the right-hand side is attained.
\end{proposition}
\begin{proof}
    This was shown in \citeauthor*{KrPa23} \cite[Theorem 4.23]{KrPa23}.  
\end{proof}

\section{Dynamic matching models} \label{sec:matching_models}
In this section we provide an application of the theory developed above, for equilibrium multipopulation matching problems, also called matching for teams, in a dynamic setting. Motivated by the work of \citeauthor*{Carlier_matching} \cite{Carlier_matching}, we introduce matching models in a dynamic framework.

\subsection{Problem formulation} Given $N \in \N$ groups of agents or workers of the same size, that we normalize to $1$. Each agent in the $i$-th group at time $t$, for $i \in \{1,\ldots,N\}$ and $t \in \{1,\ldots,T\}$, has a certain type, say $x_t^i \in \Xc^i_t$, where $\Xc_t^i$ is a given Polish space. Analogously as before we denote $\Xc^i\coloneqq \Xc^i_{1:T}$ and assume that we are given $\P^i \in \Pc(\Xc^i)$, $i \in \{1,\ldots,N\}$, representing the distributions of the types of agents in each group, which are thus assumed to be known and fixed.

\medskip Moreover, we assume that there is a group of principals with given distribution of types $\P^0 \in \Pc(\Xc^0)$ where $\Xc^0=\Xc^0_{1:T}$ is, as before, a given Polish path space.

\medskip For $i \in \{0,\ldots,N\}$, we denote by $X^i$ the canonical process on the path space $\Xc^i$ and by $\F^i\coloneqq(\Fc^{i}_t)_{t=1}^T$ the corresponding canonical $\sigma$-algebra. In accordance with the previous sections we set $\Fc^i_0=\{ \Xc^i,\emptyset \}$ and denote \begin{equation} \label{eqn:canon_proc} \X^i=\big( \Xc^i, \Fc^i_T, \F^i,\P^i,X^i \big).\end{equation}

We shall denote generic elements of $\Xc^i$ by $x^i$ and those of $\Xc^i_{1:t}$ by $x^i_{1:t}$ \emph{etc}.

\begin{remark}Let us point out that, more generally, one can, as before, consider general filtered probability spaces and, thus, general filtered processes. All results in this section remain valid, with the necessary changes having been made. Indeed, such a formulation would be of interest if the evolution of the type of each agent depends on some external randomness. For the ease of exposition, we consider the canonical processes here.
\end{remark}

We consider the situation in which every principal hires a team consisting of exactly one agent from each group to work on a task. That is, each one of the $N$ workers hired by the same principal has to work on the same task and we assume that every agent from each group enters a contract with exactly one principal. The interpretation is that the principal needs exactly one agent from each group to complete a certain task.

\medskip At each time $t \in \{1,\ldots,T\}$, we are given a set of potential tasks $\Yc_t$, which is again a Polish space, and we denote $\Yc\coloneqq\Yc_{1:T}.$ We write $d_{\Yc}(y,\tilde{y})\coloneqq \sum_{t=1}^T d_{\Yc_t}(y_t,\tilde{y}_t)$ for the metric on $\Yc.$ Further, we are given measurable functions $c^i : \Xc^i \times \Yc \longrightarrow \R,\; i\in\{0,\ldots,N\}$, where, for $i\geq 1$, $c^i(x^i,y)$ represents the cost of an agent of type $x^i$ working on task $y$. In the case $i=0$, instead, we interpret $c^0=-u$, where $u(x^0,y)$ represents the utility of a principal of type $x^0$ if her agents work on task $y$.

\medskip The aim is to find equilibrium wages (contracts) $w^i : \Yc \longrightarrow \R$, $i \in \{0,\ldots,N\}$, where, for $i\geq 1$, $w^i(y)$ is to be understood as the wage an agent in $i$-th group receives for working on task $y \in \Yc.$ The type of each agent is unobserved to the principal, and she can contract only upon the task carried out by the agent, thus facing an adverse selection situation. In the case $i=0$, we instead set \[ w^0(y) \coloneqq- \sum_{i=1}^N w^i(y),\quad y \in \Yc,\]
which is the (negative) amount that the principal pays to her agents for working on task $y$. That is, we have the relation $\sum_{i=0}^N w^i=0.$ From now on, we shall make no explicit distinction between groups of agents and principals and simply call everyone `agent' for simplicity.

\medskip Fix $i \in \{0,\ldots,N\}.$ Given a contract $w^i$ and knowing his type up to time $t$, say $x^i_{1:t} \in \Xc^i_{1:t}$, every agent chooses a task $y_t \in \Yc_t$ in an optimal way as to maximize his utility. As is standard when one seeks an equilibrium, we allow the agents to adopt mixed strategies. That is to say, agents of the same type are allowed to choose different tasks.
In mathematical terms, at every time $t \in \{1,\ldots,T\}$ agents choose a measurable kernel $K^i_{t}: \Xc^i_{1:t} \times \Yc_{1:t-1} \longrightarrow \Pc(\Yc_t).$ 
Having chosen such kernels, the expected cost of an agent in $i$-th group is then
\begin{equation} 
 \int \Big[ c^i(x^i,y_{1:T-1},\tilde{y}_T)-w^i(y_{1:T-1},\tilde{y}_T) \Big] K^i_{1}(x_1^i;\mathrm{d}y_1)K^i_{2}(x_{1:2}^i,y_1;\mathrm{d}y_2)\cdots K^i_{T}(x_{1:T}^i,y_{1:T-1};\mathrm{d}y_T) \P^i(\mathrm{d}x^i).
\end{equation}
Let  \[ \Y^\nu\coloneqq\big( \Yc, \Fc^Y_T, \F^Y,\nu,Y \big) \] be, analogously as in \eqref{eqn:canon_proc}, the canonical process on the path space $\Yc$ with distribution $\nu.$ Let us set \[\cpl(\X^i,\ast)\coloneqq\bigcup_{\nu \in \Pc(\Yc)}\cpl(\X^i,\Y^\nu),\quad{\rm and}\quad\cplc(\X^i,\ast)\coloneqq\bigcup_{\nu \in \Pc(\Yc)}\cplc(\X^i,\Y^\nu).\] It is easy to verify that a coupling $\pi^i \in \cpl(\X^i,\ast)$ belongs to 
$\cplc(\X^i,\ast)$ if and only if it admits a disintegration
\[ \pi^i(\mathrm{d}x^i,\mathrm{d}y)=K^i_{1}(x^i_1;\mathrm{d}y_1)K^i_{2}(x^i_{1:2},y_1;\mathrm{d}y_2)\cdots K^i_{T}(x^i_{1:T},y_{1:T};\mathrm{d}y_T) \P^i(\mathrm{d}x^i) \] for such kernels introduced above. We thus equivalently minimize the cost over all couplings \[\pi^i \in  \cplc(\X^i,\ast)\] and the expected cost in $i$-th group for such a coupling is then
\begin{equation} 
 \int \Big[ c^i(x^i,y_{1:T})-w^i(y_{1:T}) \Big] \pi^i(\mathrm{d}x^i,\mathrm{d}y).
\end{equation} We therefore define the optimal value associated to a given wage $w^i$ by
\begin{align*} V^i(w^i)&\coloneqq \inf \bigg\{ \int \Big[ c^i(x^i,y_{1:T})-w^i(y_{1:T}) \Big] \pi^i(\mathrm{d}x^i,\mathrm{d}y) \,\bigg|\, \pi^i \in \cplc(\X^i,\ast) \bigg\}.
\end{align*}

\medskip Furthermore, as the principal hires a team of agents, in order to complete a task, the whole team has to work together on the said task. Consequently, at an equilibrium we impose that the distributions of the tasks chosen by the agents between groups coincide. That is to say, there exists a distribution of tasks $\nu \in \Pc(\Yc)$ such that
${\rm proj}^2_{\#}\pi^i = \nu$ for all $i \in \{0,\ldots,N\}$. We can also interpret this condition as imposing that the distributions of the demand and the supply for agents match.

\medskip We are thus looking for an equilibrium consisting of measurable functions $w^i : \Yc \longrightarrow \R$, $i \in \{0,\ldots,N\}$, a distribution $\nu \in \Pc(\Yc)$ and couplings $\pi^i \in \cplc(\X^i,\Y^\nu)$, $i \in \{0,\ldots,N\}$, satisfying the following definition, where,  for brevity, we set
\[ \Theta=\Big\{ (w^{0:N},\nu, \pi^{0:N}) \,\Big\vert\, w^i : \Yc \longrightarrow \R\; {\rm is}\;{\rm measurable},\; \nu \in \Pc(\Yc),\;  \pi^i \in \cplc(\X^i,\Y^\nu),\; i \in \{0,\ldots,N\}  \Big\}.\] 

\begin{definition} \label{def:equi} We say that $(w^{0:N},\nu, \pi^{0:N}) \in \Theta$ forms an equilibrium, if the following conditions hold:
\begin{enumerate}[label = (\roman*)] 
\item \label{def:equi_clearing} Clearing:  $\sum_{i=0}^N w^i(y)=0,\; y \in \Yc;$ 
\item \label{def:equi_opti} Optimality:  for every $i \in \{0,\ldots,N \}$ we have \begin{equation*} \int \big[ c^i(x^i,y_{1:T-1},\tilde{y}_T)-w^i(y_{1:T-1},\tilde{y}_T) \big] \pi^i(\mathrm{d}x^i,\mathrm{d}y)=V_0^i(w^i).
\end{equation*}
\end{enumerate}
\end{definition}

\subsection{Existence of equilibria}
In this section, we link equilibria to solutions to a causal barycenter problem and the corresponding dual. This in turn, using \Cref{existCausBary} and \Cref{thm:existbary_dual}, gives existence of an equilibrium under appropriate assumptions. 

\medskip First, we establish the following optimality verification theorem.

\begin{lemma} \label{lem:opti} Let $f^i \in \L^1(\P^i)$ and $w^i: \Yc \longrightarrow \R$ be measurable such that for every $\nu \in \Pc(\Yc)$ there is $G^i \in \mathfrak{G}^{i,{\rm c}}$ such that it holds $f^i(x^i)\leq c^i(x^i,y)-w^i(y)+G^i(x^i,y)$ $\cplc(\X^i,\Y^\nu)$--q.s. Assume that there exist $\nu^\star \in \Pc(\Yc)$, $\pi^{i,\star} \in  \cplc(\X^i,\Y^{\nu^\star})$ and $G^{i,\star} \in \mathfrak{G}^{i,{\rm c}}$ such that
$f^i(x^i)\leq c^i(x^i,y)-w^i(y)+G^{i,\star}(x^i,y)$ holds $\cplc(\X^i,\Y^{\nu^\star})$--q.s.\ and with equality $\pi^{i,\star}$--a.s. Then $\pi^{i,\star}$ solves $V^i(w^i).$
\end{lemma}

\begin{proof}
First, we note that $V^i(w^i)$ is, by definition, a specific type of causal barycenter with only one marginal $\P^i$ and cost function $\tilde{c}\coloneqq c^i-w^i$. By the weak duality result, which can be derived analogously as in \Cref{rem:weak_dual_bary}, this yields
\begin{multline*}
    \sup \bigg\{ \int f^i \mathrm{d}\P^i \,\bigg|\,f^i \in \L^1(\P^i),\; \forall \nu \in \Pc(\Yc),\;\exists G^i \in \mathfrak{G}^{i,{\rm c}}: \\ f^i(x^i)\leq c^i(x^i,y)-w^i(y)+G^i(x^i,y)\enspace\cplc(\X^i,\Y^\nu){\rm \text{--q.s.}}  \bigg\} \leq V^i(w^i).
\end{multline*}
Thus, if $\nu^\star,$ $G^{i,\star} \in \mathfrak{G}^{i,{\rm c}}$ and $\pi^{i,\star} \in \cplc(\X^i,\Y^{\nu^\star})$ are such that
$ f^i(x^i)\leq c^i(x^i,y)-w^i(y)+G^{i,\star}(x^i,y)$ $\cplc(\X^i,\Y^{\nu^\star})$--quasi-surely and with equality $\pi^{i,\star}$--almost surely, then it is clear that the inequality above is actually an equality and that $ f^i$, resp.\ $\pi^{i,\star}$, solves the respective problem. 
\end{proof}  

\begin{remark} Let us comment on the result presented in {\rm \Cref{lem:opti}}. Similarly as in {\rm \cite[Corollary 3.17]{KrPa23}}, we can connect the martingale $f^i-G^i$ to the value process corresponding to the optimal control problem of agents in the $i$-th group. Vaguely speaking, {\rm \Cref{lem:opti}} can be interpreted as a verification theorem for the corresponding optimal control problem.
\end{remark}

We have the following existence result, where we show existence of an equilibrium using solutions to the barycenter problem as well as the corresponding dual.

\begin{theorem} \label{thm:equi} Assume that {\rm\Cref{assAdaptBary}} holds with $B_{c^{1:N}}^{\rm bc}(\Y^0)<\infty$ replaced by $B_{c^{1:N}}^{\rm c}(\Y^0)<\infty.$ For every $ i \in \{1,\ldots,N\},$ let $c^i$ be measurable and lower-bounded, $k : \Yc \longrightarrow \R$ be bounded on compacts, and $\ell^i \in \L^1({\rm Law}_{\P^i}(X^i))$ be such that $c^i(x, y) \leq \ell^1(x) + k(y), (\omega^i, y) \in \Omega^i \times \Yc.$ Then there exists an equilibrium $(w^{0:N,\star},\nu^\star, \pi^{0:N,\star}) \in \Theta.$
\end{theorem}

\begin{proof} The assumptions imply that the problem 
\begin{equation}\label{egn:bary_equi} \inf_{\nu \in \Pc(\Yc)}\sum_{i=1}^N \Cc\Wc_{c^i}(\X^i,\Y^{\nu})\end{equation} is attained and admits the dual representation, see \eqref{eqn:dual_caus_bary}, which is attained as well. In particular, there exist $f^{i,\star} \in \L^1(\P^i),$ $i \in \{0,\ldots,N\},$ such that for every $\nu \in \Pc(\Yc),$ there exist $G^{i,\nu} \in \mathfrak{G}^{i,{\rm c}}$, measurable $w^{i,\nu} : \Yc \longrightarrow \R$, such that
$ f^{i,\star}(x^i)\leq c^i(x^i,y)-w^{i,\nu}(y)+G^{i,\nu}(x^i,y)$, $\sum_{i=1}^N w^{i,\nu}(y)=0$ $\cplc(\X^i,\Y^\nu)$--quasi-surely, and $(f^{i,\star})_{i=1}^N$ attains the dual problem. Let $\nu^\star$ be a solution to \eqref{egn:bary_equi} and $\pi^{i,\star}\in \cplc(\X^i,\Y^{\nu^\star})$ be the corresponding optimal transport plans. It is then easy to see that $(w^{0:N,\nu^\star},\nu^\star, \pi^{0:N,\star}) \in \Theta$ satisfies \Cref{def:equi}.$\ref{def:equi_clearing}$. It thus suffices to verify that \Cref{def:equi}.$\ref{def:equi_opti}$ is met. To that end, it is straightforward to verify that, by optimality of $\nu^\star$ and $\pi^{i,\star}$, for every $i \in \{0,\ldots,N\}$, we have
\[ f^{i,\star}(x^i)= c^i(x^i,y)-w^{i,\nu^\star}(y)+G^{i,\nu^\star}(x^i,y)\quad \pi^{i,\star}{\rm \text{--almost surely.}}\]
Invoking \Cref{lem:opti}, we obtain that \Cref{def:equi}.$\ref{def:equi_opti}$ is satisfied. This concludes the proof.
\end{proof} 

\appendix
\section{Appendix}

The appendix collects technical lemmata related to multicausal couplings and (bi-)causal transport.

\begin{lemma} \label{lem:compactness}
    The sets $\cplbc(\X^1,\X^2)$, $\cplc(\X^1,\X^2),$ $\cplmc(\X^1,\ldots,\X^N)$ are compact in the weak topology.
\end{lemma}

\begin{proof}
    The first two cases were treated in \citeauthor*{EcPa22} \cite[Lemma A.3]{EcPa22}. The multicausal case can be shown analogously.
\end{proof}

\begin{lemma} \label{lem:mcrestriction}
    Let $\pi \in \cplmc(\X^1,\ldots, \X^N)$ and $I \subseteq \{1,\ldots,N\}$.
    Then $\proj^I_\# \pi \in \cplmc( (\X^i)_{i \in I})$.
\end{lemma}
\begin{proof}
    Let $v \in \R^I$ be a vector. We write $\overline{\Gc}_v \coloneqq \bigotimes_{i \in I} \Fc^i_{v_i}$ for simplicity. Fix an arbitrary index $j \in I$ and denote $\tilde{\pi}\coloneqq \proj^I_\# \pi.$ Let further $U,V$ be $\overline{\Gc}_{T e_j}$-measurable and $\overline{\Gc}_{(t,\ldots,t)}$-measurable, respectively, and bounded. 
    Then $U \circ \proj^I$ and $V \circ \proj^I$ are $\overline{\Fc}_{T e_j}$- and $\overline{\Fc}_t$-measurable, respectively.
    It follows that $\pi$--almost surely the following holds
    \begin{align*}
        \E_{\tilde \pi}[UV| \overline{\Gc}_{t e_j}] \circ \proj^I &=
        \E_\pi[(UV) \circ \proj^I | \overline{\Fc}_{t e_j}] 
        = \E_\pi[U \circ \proj^I |\overline{\Fc}_{t e_j}] \cdot \E_\pi[V \circ \proj^I |\overline{\Fc}_{t e_j}]\\
        &= (\E_{\tilde \pi}[U|\overline{\Gc}_{t e_j}] \cdot \E_{\tilde \pi}[V |\overline{\Gc}_{t e_j}]) \circ \proj^I,
    \end{align*}
    where the second equality is due to multicausality of $\pi$. Hence, we have that $\tilde \pi$--a.s.
    \[
        \E_{\tilde \pi}[UV | \overline{\Gc}_{te_j}] = \E_{\tilde \pi}[U|\overline{\Gc}_{t e_j}] \E_{\tilde \pi}[V |\overline{\Gc}_{t e_j}],
    \]
    which shows that $\overline{\Gc}_{Te_j}$ is conditionally independent of $\overline{\Gc}_t$ given $\overline{\Gc}_{t e_j}$. This concludes the proof.
\end{proof}

\begin{lemma} \label{disintegr} 
    Let $\pi \in \cpl(\X^1,\ldots,\X^N)$.
    Then $\pi$ is multicausal if and only if there are $\overline{\Fc}_{t-1}$-measurable kernels $K_t: \overline{\Omega}_{1:t-1} \longrightarrow \Pc(\overline{\Omega}_{t})$,  $t\in\{2,\ldots,T-1\}$ and $K_1 \in \Pc(\overline{\Omega}_1)$ such that
    \begin{equation}
        \label{eq:lem.disintegr}
        \pi( \,\cdot\, | \overline{\Fc}_{t-1})|_{\overline{\Fc}_t} = K_t \;\text{with}\;K_t(\overline{\omega}_{1:t-1}) \in \cpl(\P^N_{t,\omega^\smallfont{i}_{\smallfont{1}\smallfont{:}\smallfont{t}\smallfont{-}\smallfont{1}}}, \ldots, \P^N_{t,\omega^\smallfont{i}_{\smallfont{1}\smallfont{:}\smallfont{t}\smallfont{-}\smallfont{1}}}),\;{\rm resp.}\;K_1 \in \cpl(\P_1^1,\ldots,\P_1^N)\;\pi{\text{\rm --a.s.}}
    \end{equation}
    In this case, we have $\pi(\mathrm{d} \overline{\omega}) = K_1(\mathrm{d} \overline{\omega}_{1}) \otimes K_{2}(\overline{\omega}_{1};\mathrm{d} \overline{\omega}_{2}) \otimes \ldots \otimes K_T(\overline{\omega}_{1:T-1}; \mathrm{d} \overline{\omega}_T).$
\end{lemma}
\begin{proof}
    As we have assumed that all considered spaces are Polish, every $\pi \in \Pc(\overline{\Omega})$ admits a regular conditional probability $\pi( \,\cdot\, | \overline{\Fc}_t)$ given $\overline{\Fc}_t$.
    Then we have that $\overline{\Fc}_{T e_i}$ is conditionally independent of $\overline{\Fc}_t$ given $\overline{\Fc}_{te_i}$ under $\pi$ if and only if
    \begin{equation}
        \label{eq:lem.disintegr.1}
        \proj^i_\# \pi( \,\cdot\, | \overline{\Fc}_t) = \proj^i_\# \pi(\,\cdot\, | \overline{\Fc}_{te_i}) = \P^i(\;\cdot\; | \Fc^i_t) \quad \pi\text{--a.s.,}
    \end{equation}
    where the last equality is due to $\proj^i_\# \pi = \P^i$.
    Hence, by potentially (for every $t$) modifying the regular conditional probability $\pi(\,\cdot\, | \overline{\Fc}_t)$ on a $\pi$-null set, we deduce from \eqref{eq:lem.disintegr.1} that $\pi \in \cplmc(\X^1,\ldots,\X^N)$ if and only if \eqref{eq:lem.disintegr} holds. The last assertion follows by the tower property. Indeed, for any bounded $\overline{\Fc}_T$-measurable random variable $V$, 
       \begin{multline*}
        \E_\pi[V] = \E_\pi[\E_\pi[V | \overline{\Fc}_{T-1}] ]=\E_\pi[\E_\pi[\ldots \E_\pi [V \vert \overline{\Fc}_{T-1} ] \ldots | \overline{\Fc}_{2}] | \overline{\Fc}_{1}] \\= \int \ldots \int V(\overline{\omega}) \, K_{T}(\overline{\omega}_{1:T-1}; \mathrm{d}\overline{\omega}_T) \ldots K_1(\mathrm{d} \overline{\omega}_1),
    \end{multline*}
which concludes the proof.
\end{proof}

\begin{remark}
    As a consequence of {\rm\Cref{disintegr}}, we emphasize that when $(K_t)_{t = 1}^T$ is a family of kernels where, for every $t\in\{1,\ldots,T\}$, $K_t$ is $\overline{\Fc}_{t-1}$-measurable and 
    \[
        K_t(\overline{\omega}_{1:t-1}) \in \cpl(\P^N_{t,\omega^\smallfont{i}_{\smallfont{1}\smallfont{:}\smallfont{t}\smallfont{-}\smallfont{1}}}, \ldots, \P^N_{t,\omega^\smallfont{i}_{\smallfont{1}\smallfont{:}\smallfont{t}\smallfont{-}\smallfont{1}}}),\;t \in \{2,\ldots,T\},\quad{\rm and}\quad K_1 \in \cpl(\P^1_1,\ldots,\P^N_1),
    \]
    then the probability measure
    \begin{equation}
        \label{eq:def_via_disintgr}
        \pi(\mathrm{d}\overline{\omega}) \coloneqq K_1(\mathrm{d} \overline{\omega}_1) \otimes K_{2}(\overline{\omega}_{1};\mathrm{d}\overline{\omega}_{2}) \otimes \ldots \otimes K_T(\overline{\omega}_{1:T-1};\mathrm{d}\overline{\omega}_T)
    \end{equation}
    is contained in $\cplmc(\X^1,\ldots,\X^N)$.
    This follows directly as by construction $\pi \in \cpl(\X^1,\ldots,\X^N)$ and therefore all assumptions of {\rm Lemma~\ref{disintegr}} are satisfied.
    
    \medskip
    In the special case when $(\Omega^i,\Fc^i,\F^i,\P^i)$ coincides with $(\Xc,\Fc^X_T,\F^X,\mu^i)$ for some $\mu^i \in \Pc(\Xc)$, where $(\Xc,\Fc^X_T,\F^X)$ is the canonical path space with its filtration, and $X^i$ is the canonical process on $\Xc$ for every $i \in\{ 1,\ldots,N\}$, {\rm \Cref{disintegr}} can be paraphrased as follows:
    Let $\pi \in \Pc(\pmb{\Xc})$ be a coupling with marginals $\proj^i_{\#} \pi = \mu^i$ and write $\mu^i_{t,x^\smallfont{i}_{\smallfont{1}\smallfont{:}\smallfont{t}\smallfont{-}\smallfont{1}}}$  for the disintegration of ${\proj^{1:t}}_\# \mu^i$ with respect to ${\proj^{1:t-1}}_\#\mu^i$.
    Then $\pi$ is multicausal if and only if
    \[
        \pi(\mathrm{d} x^{1:N}_{1:T}) = \pi_1(\mathrm{d} x^{1:N}_1) \otimes \tilde K_2(x^{1:N}_1; \mathrm{d}x^{1:N}_2) \otimes \ldots \otimes \tilde K_T(x^{1:N}_{1:T-1}; \mathrm{d}x^{1:N}_T),
    \]
    where, for every $t \in \{2,\ldots,T\}$, $\tilde K_t: \Xc^N_{1:t-1} \longrightarrow \Pc(\Xc^N_t)$ is a measurable kernel with $\tilde K_t(x^{1:N}_{1:t-1}) \in \cpl(\mu^1_{t,x^\smallfont{i}_{\smallfont{1}\smallfont{:}\smallfont{t}\smallfont{-}\smallfont{1}}}, \ldots,\mu^1_{t,x^\smallfont{i}_{\smallfont{1}\smallfont{:}\smallfont{t}\smallfont{-}\smallfont{1}}})$ and $\pi_1 \in \cpl(\mu^1_1,\ldots,\mu^N_1)$.
\end{remark}

\begin{lemma}  \label{lem:mcgluing}
    Let $M \in \{1,\ldots,N\}$, $\pi \in \cplmc(\X^1,\ldots,\X^M)$ and $\gamma  \in \cplmc(\X^M,\ldots, \X^N)$. Let $K: \Omega^M \longrightarrow \Pc(\Omega^{M+1:N})$ be a measurable kernel such that $\gamma(\mathrm d \omega^M,\ldots\mathrm d \omega^N)=\gamma(\mathrm d \omega^M) \otimes K(\omega^M;\mathrm d \omega^{M+1},\ldots\mathrm d \omega^N).$
    Then $\pi \otimes K \in \cplmc(\X^1,\ldots,\X^{N})$
    where \[ ( \pi \otimes K )(\mathrm{d}\overline{\omega})\coloneqq K(\omega^M;\mathrm{d}\omega^{M+1,N})\pi(\mathrm{d}\omega^{1:M}).\]
\end{lemma} 

\begin{proof}
    We have to show for every $t \in \{1,\ldots,T\}$ and $i \in \{1,\ldots,N\}$ that under $\pi \otimes K$
    \begin{equation}
        \label{eq:lem.mcgluing.1}
         \overline{\Fc}_{Ne_i} \text{ is conditionally independent of } \overline{\Fc}_{t},\; \text{given }\overline{\Fc}_{t e_i}.
    \end{equation}
    Due to symmetry, we may assume without loss of generality that $i \in \{1,\ldots,M\}$. Indeed, it is easy to verify that if we interchange the roles of $\pi$ and $\gamma,$ we obtain the same measure. In other words, we have $\gamma \otimes K^\pi=\pi \otimes \gamma,$ where $\pi=({\rm proj}^{M}_{\#}\gamma ) \otimes K^\pi,$ defines the same measure as $\pi \otimes K$.
    \medskip By \cite[Lemma 24 (1)]{pa22} we have for every $t \in \{1,\ldots,T\}$ that under $\pi \otimes K$
    \begin{equation}
        \label{eq:lem.mcgluing.2}
    \overline{\Fc}_{Nv} \text{ is conditionally independent of } \overline{\Fc}_{tw},\; \text{given }\overline{\Fc}_{t v},
    \end{equation}
    where $v, w \in \{0,1\}^N$ with $v_i + w_i = 1,$ $i \in \{1,\ldots,N\},$ and $v_i = 1$ if $i \le M$ and $0$ otherwise.
    Further, as $\pi$ is multicausal we also have
    \begin{equation}
        \label{eq:lem.mcgluing.3}
        \overline{\Fc}_{Ne_i} \text{ is conditionally independent of }\overline{\Fc}_{tv},\; \text{given }\overline{\Fc}_{te_i}.
    \end{equation}
    By invoking \cite[Theorem 8.12]{Kall02}, we immediately
 derive \eqref{eq:lem.mcgluing.1} from \eqref{eq:lem.mcgluing.2} and \eqref{eq:lem.mcgluing.3}.
\end{proof}

\begin{lemma} \label{lemma:glueing} Let $\Y \in {\rm FP}(\Yc)$ and $\pi^i \in \cplbc(\X^i,\Y),\,i\in \{1,\ldots,N\}.$ Then there exists a probability measure  $\Q \in \Pc (\overline{\Omega} \times\Omega^\Y, \overline{\Fc}_T \otimes \Fc^\Y_T )$ such that
\begin{enumerate}[label = (\roman*)]
\item \label{it:lem.gluing.1} $(\mathrm{proj}^i,\mathrm{proj}^{N+1})_\#\Q=\pi^i;$
\item \label{it:lem.gluing.2} $\Q \in \cplmc(\X^1,\ldots,\X^N,\Y)$.
\end{enumerate}
In particular, ${\rm proj}^{ 1:N }_\#\Q \in \cplmc(\X^1,\ldots,\X^N)$.
\end{lemma}

\begin{proof}
    For notational simplicity, we write $\omega^\Y$ for the elements in $\Omega^\Y$.
    Let, for every $\pi^i$, $i \in \{ 1,\ldots,N \}$, $K^i : \Omega^\Y \longrightarrow \Pc(\Omega^i)$ be a measurable kernel satisfying $ \pi^i(\mathrm{d}\omega^i,\mathrm{d}\omega^\Y)=K^i(\omega^\Y;\mathrm{d}\omega^i) \P^\Y(\mathrm{d}\omega^\Y).$
    Define
    \[ \Q(\mathrm{d}\overline{\omega},\mathrm{d}\omega^\X)\coloneqq  \Big(\bigotimes_{i=1}^N K^i(\omega^\Y;\mathrm{d}\omega^i) \Big) \otimes \P^\Y(\mathrm{d}\omega^\Y). \]
    Clearly, $\Q$ satisfies $\ref{it:lem.gluing.1}$.
    Next, $\ref{it:lem.gluing.2}$ follows from a simple induction using Lemma \ref{lem:mcgluing}, which asserts that multicausality is preserved under gluing.
    The last statement is a consequence of Lemma~\ref{lem:mcrestriction}. \end{proof}

\begin{lemma} \label{lem:multic_ip} Let $(\X^1, \ldots, \X^N) \in {\rm FP}^N.$ Then the following are true.
\begin{enumerate}[label = (\roman*)]
    \item \label{itm:multic_ip1} If $\pi \in \cplmc(\X^1,\ldots,\X^N),$ then ${\rm Law}_{\pi}({\rm ip}(\X^1),\ldots,{\rm ip}(\X^N)) \in \cplmc({\rm ip}(\X^1),\ldots,{\rm ip}(\X^N)).$
    \item \label{itm:multic_ip2} Conversely, if $\gamma \in \cplmc({\rm ip}(\X^1),\ldots,{\rm ip}(\X^N)),$ then there exists $\pi \in \cplmc(\X^1,\ldots,\X^N)$ such that $\gamma={\rm Law}_{\pi}({\rm ip}(\X^1),\ldots,{\rm ip}(\X^N)).$
\end{enumerate}
\end{lemma} 
\begin{proof}
Let us first show $\ref{itm:multic_ip1}$. Since $({\rm id},{\rm ip}(\X^i))_{\#} \P^i \in \cplbc(\X^i,{\rm ip}(\X^i))$ by \cite[Theorem 3.9.$(i)$]{BaBePa21}, the result follows from repeatedly using \Cref{lem:mcgluing}.

\medskip As for $\ref{itm:multic_ip2},$ consider the disintegration $\P^i(\mathrm{d}\omega^i)=\mu^i(\mathrm{d}z) \otimes \kappa^i(z;\mathrm{d}\omega^i),$ where $\mu^i\coloneqq {\rm Law}_{\P^i}({\rm ip}(\X^i)) \in \Pc(\Zc),$ and define \[ \pi(\mathrm{d}\omega^1,\ldots,\mathrm{d}\omega^n)\coloneqq \gamma(\mathrm{d}z^1,\ldots,\mathrm{d}z^N)\otimes \prod_{i=1}^N \kappa^i(z^i;\mathrm{d}\omega^i). \]
Since $({\rm id},{\rm ip}(\X^i))_{\#} \P^i \in \cplbc(\X^i,{\rm ip}(\X^i))$ and $\gamma \in \cplmc ({\rm ip}(\X^{1}),\ldots, {\rm ip}(\X^{N}))$, we get $\pi \in \cplmc(\X^1,\ldots,\X^N)$ by analogous arguments as in \Cref{lemma:glueing}. 
\end{proof}

\begin{lemma} \label{lem:comp_multic} Let $\Kc_i \subseteq {\rm FP}$ for $i \in \{1,\ldots,N\}$ be compact. Then the set
\[ \K \coloneqq\Big\{ \cplmc ({\rm ip}(\X^1),\ldots,{\rm ip}(\X^N)) \in \Pc(\Zc^N)\, \Big\vert\,\X^1 \in \Kc_1,\ldots,\X^N \in \Kc_N \Big\} \] is compact.   
\end{lemma}

\begin{proof} 
Let \[\gamma^n=\cplmc({\rm ip}(\X^{1,n}),\ldots,{\rm ip}(\X^{N,n})) \in \K \text{ for some }\X^{1,n} \in \Kc_1,\ldots, \X^{N,n} \in \Kc_N,\;n \in \N,\] be a sequence in $\K$. Then, for every $i \in \{1,\ldots,N\}$, up to passing to a subsequence, the sequence $(\X^{i,n})_{n \in \N}$ admits a limit $\X^{i,\infty} \in \Kc_i$ by compactness of $\Kc_i.$ It follows that for every $i \in \{1,\ldots,N\}$ the set $\{ {\rm Law}_{\P^{i,n}}({\rm ip}(\X^{i,n}))\,\vert\, n \in \N \cup \{\infty\}\}$ is tight, and so the sequence $\gamma^n$ admits a limit up to choosing a subsequence, say $\gamma^\infty,$ which belongs to $\cpl({\rm ip}(\X^{1,\infty}),\ldots,{\rm ip}(\X^{N,\infty}))$ by standard arguments using the isometry \cite[Theorem 3.10]{BaBePa21}; see \emph{e.g.}\ the proof of \cite[Theorem 1.7.2]{PaZe20}. We now show multicausality of $\gamma^\infty$. For $\gamma \in \cpl ({\rm ip}(\X^{1}),\ldots, {\rm ip}(\X^{N})),$ we have that $\gamma \in \cplmc ({\rm ip}(\X^{1}),\ldots, {\rm ip}(\X^{N}))$ if and only if for every $t \in \{1,\ldots,T-1\}$ and $i \in \{1,\ldots,N\}$ and for every $\varphi$ and $f$ countinuous and bounded it holds \begin{multline}   \int_{\Zc^N} \Big[ \varphi(z^1_{1:t},\ldots,z^N_{1:t}) f(z^i_{t+1:T}) \Big]\gamma(\mathrm{d}z^1,\ldots,\mathrm{d}z^n)\\
=\int_{\Zc_{1:t}^N} \bigg[ \varphi(z^1_{1:t},\ldots,z^N_{1:t}) \int_{\Zc_{t+1}} f(z_{t+1:T}^i) p^i_{t}(\mathrm{d} z_{t+1:T}^i) \bigg]\gamma(\mathrm{d}z^1,\ldots,\mathrm{d}z^n). \label{eqn:multic_test} \end{multline}
As the functions in the integrals on both sides are continuous and bounded functions on $\Zc$ and the equality \eqref{eqn:multic_test} is satisfied by $\gamma^n$ by \Cref{lem:multic_ip}.$\ref{itm:multic_ip1}$, we conclude that $\gamma^\infty \in \cplmc ({\rm ip}(\X^{1}),\ldots, {\rm ip}(\X^{N}))$ by passing to the limit on both sides in \eqref{eqn:multic_test}.
\end{proof}

\begin{lemma} \label{lem:topology} Let $\Xc=(\Xc,\tau)$ be a Polish topological space and let $(\Yc_n)_{n \in \N}$ be a sequence of Polish spaces. Further, let $(f_n)_{n \in \N}$ be a sequence of measurable functions $f_n : \Xc \longrightarrow \Yc_n$, $n \in \N.$ Then, there exists a refinement of the topology $\tau$, say $\tilde{\tau} \supset \tau$, such that the following holds:
\begin{enumerate}[label = (\roman*)]
\item $f_n$ is continuous with respect to $\tilde{\tau}$ for every $n \in \N$;
\item the topological space $(\Xc,\tilde{\tau})$ is Polish;
\item the Borel $\sigma$-algebras generated by $\tau$ and $\tilde{\tau}$ coincide. 
\end{enumerate}
\end{lemma}
\begin{proof} See \citeauthor*{Ke95} \cite[Exercise 13.12 (ii)]{Ke95}. \end{proof}

\bibliography{refs}

\end{document}